\documentclass[11pt,reqno,a4paper]{amsart}
\usepackage{amsmath,amscd,amstext,amsthm,amsfonts,latexsym}
\usepackage[dvips]{graphicx}
\usepackage{graphicx, graphics}
\usepackage{enumerate}
\usepackage{amscd, color, enumerate}

\usepackage{amsmath,amssymb,amsthm}
\usepackage{color}

\theoremstyle{plain}
\newtheorem{theorem}{Theorem}[section]
\newtheorem{proposition}[theorem]{Proposition}
\newtheorem{lemma}[theorem]{Lemma}
\newtheorem{corollary}[theorem]{Corollary}
\newtheorem{maintheorem}{Theorem}

\theoremstyle{definition}
\newtheorem{remark}[theorem]{Remark}
\newtheorem{example}[theorem]{Example}
\newtheorem{definition}[theorem]{Definition}

\numberwithin{equation}{section}

\title[Mean dimension and metric mean dimension for non-autonomous    systems]{Mean dimension and metric mean dimension for non-autonomous  dynamical systems}
\author{Fagner B. Rodrigues \quad and \quad  Jeovanny Muentes Acevedo}

\begin{document}

\address{Fagner B. Rodrigues, Departamento de Matem\'atica, Universidade Federal do Rio Grande do Sul, Brazil}
\email{fagnerbernardini@gmail.com}

\address{Jeovanny de Jesus Muentes Acevedo, Facultad de Ciencias B\'asicas,  Universidad Tecnol\'ogica de  Bol\'ivar, Cartagena de Indias - Colombia}
\email{jmuentes@utb.edu.co}

\begin{abstract}
In this paper we extend the definitions  of mean dimension and metric mean dimension for non-autonomous  dynamical systems. We show some properties of this extension and furthermore some applications to the mean dimension and metric mean dimension of single continuous maps. 
 \end{abstract}



\date{\today}
\maketitle


\section{Introduction}

In   the late 1990's, M. Gromov in \cite{Gromov} introduced the notion of mean dimension for a topological dynamical system $(X,\phi)$
($X$ is a compact topological space and $\phi$ is a continuous map on $X$), which is, as well as the topological entropy, an invariant under conjugacy.    In \cite{lind}, Lindenstrauss and Weiss showed  that the mean dimension is zero if the topological dimension
of $X$ is finite. They  gave some examples where the mean dimension is positive. For instance, they proved that the mean dimension of   $(([0,1]^m)^{\mathbb{Z}}, \sigma)$, where $\sigma$ is the two-sided full shift map on $([0,1]^m)^{\mathbb{Z}}$,  which has infinite topological entropy,   is equals to $m$ and that any non-trivial factor of  $( ([0,1]^m)^{\mathbb{Z}},\sigma)$
has positive mean dimension. 

\medskip

Given a dynamical system  $(X,\phi)$, an interesting question related to such a system is the following: under what conditions is it possible to imbed  such a system in the shift space $(([0,1]^\mathbb{N})^{\mathbb{Z}}, \sigma)$? That is, what properties the system must have to guarantee the existence of a continuous map $i: X\to ([0,1]^\mathbb{N})^{\mathbb{Z}}$  satisfying $\sigma \circ i = i \circ \varphi$?  In \cite{lind} the authors proved that 
 a necessary condition for an invertible system $(X,\phi)$ to be embedded in 
$(([0,1]^m)^{\mathbb{Z}},\sigma)$   is   that $\text{mdim}(X,\phi)\leq m$, where  $\text{mdim}(X,\phi)$ denotes the mean dimension of the system $(X,\phi)$.
In \cite{lind3}  it was  proved that if $ (X, \phi)$ is an invertible system which is an extension of a minimal system, and $K$ is  a convex set with non-empty interior such that $\text{mdim} (X,\phi)< \text{dim} K /36$, then  $(X, \phi)$ can be embedded in the shift space $(K^{\mathbb{Z}},\sigma)$. In particular, if $\text{mdim} (X,\phi)< m/36$, then  $(X, \phi)$ can be embedded in $(([0,1]^m)^{\mathbb{Z}},\sigma)$. More recently, Gutman and Tsukamoto  \cite{Gutman} showed that, that if $(X, \phi)$ is a minimal system with $\text{mdim}(X, \phi) <N/
2$ then we can embed it in
$(([0, 1]^{N})^{\mathbb{Z}},\sigma)$. In \cite[Theorem 1.3]{lind4}, Lindenstrauss and Tsukamoto 
 constructed a minimal system with mean dimension equal to $N/2$ which cannot be
embedded into $(([0, 1]^{N})^{\mathbb{Z}},\sigma)$, showing that the constant $N/2$ obtained in \cite{Gutman} is optimal. 

\medskip

The notion of metric mean dimension for a dynamical system $\phi:(X,d)\to (X,d)$ was   introduced in \cite{lind},
where $(X,d)$ is a compact metric space with metric $d$  and $\phi$ is a continuous map. It refines the topological entropy for systems with infinite entropy,
which, in the case of a manifold of dimension greater than one,  form a residual subset of the set consisting of   homeomorphisms   defined on the manifold (see \cite{Yano}). In fact,
 every system with finite topological entropy has metric mean dimension equals to zero and
for any metric $d'$ equivalent to $d$ on $X$ one has $\text{mdim}(X,\phi)\leq \text{mdim}_M(X,\phi,d')$, where $\text{mdim}_M(X,\phi,d')$
denotes the metric mean dimension of $(X,\phi)$ with respect to  $d'$ (see \cite{lind2}, \cite{lind}). The metric mean dimension depends on the metric   $d$, therefore it is not a  topological invariant. However, for a metrizable topological space $X$,  $\text{mdim}_M(X,\phi)=\inf _{d'}\text{mdim}_M(X,\phi,d')$ is invariant under topological conjugacy, where the infimum  is taken over all the metrics on $X$ which induce the topology on $X$. In  \cite{lind2}, Theorem 4.3, the author  proved  that if $(X,\phi)$  is an extension of a minimal system, then there exists a metric $d^{\prime}$ on $X$, equivalent to $d$, 
such that $\text{mdim}(X,\phi) = \text{mdim}_{M}(X,\phi, d^{\prime})$.  

\medskip

B. Kloeckner (\cite{BK}) studied the dynamical system $(\mathcal{P}(\mathbb S^1),\Phi_{d\sharp})$, where
$\mathcal{P}(\mathbb S^1)$ is the space of probability measures on the circle $\mathbb S^1$ and
$\Phi_{d\sharp}$ is the push-forward map induced by a $d$-expanding map $\Phi_d:\mathbb S^1\to\mathbb S^1$.
The author shows   if we  take the Wasserstein metric with cost function
$|\cdot |^p$ ($p\in[1,\infty)$) on $\mathcal{P}(\mathbb S^1)$, denoted by $\mathcal{W}_p$,   then
$\text{mdim}_M(\mathcal{P}(\mathbb S^1),\Phi_{d\sharp},\mathcal{W}_p)\geq p(d-1)$. H. Lee (in \cite{HL})
introduced the mean dimension for continuous actions of countable sofic groups on compact metrizable spaces and proved that, in this setting, the  mean dimension is an important tool for distinguishing continuous actions of countable
sofic groups with infinite entropy.

\medskip

A \emph{non-autonomous dynamical system} (or a sequential dynamical system) is a sequence $\textit{\textbf{f}}=(f_n)_{n=1}^{\infty}$  of continuous maps $f_n : X_n \to X_{n+1}$, where $X_n$ is a compact topological space for every $n\in \mathbb N$.
In the last two decades, several authors have tried to extend some results that are valid for autonomous systems for the non-autonomous case. Kolyada and Snoda in \cite{K-S} introduced the notion  of topological  entropy for this setting and proved that, just as in the case of autonomous systems, it is an invariant under equiconjugacy and furthermore  that it is concentrated in the non-wandering set of the  dynamics   (see \cite{K-S} and \cite{JeoCTE}). In a more recent work, Freitas et al \cite{freitas} have analyzed the existence of Extreme Value Laws in this setting. In \cite{Sta} Stadlbauer guarantees, under appropriate conditions, the existence of a spectral gap for transference operators associated with sequential systems.

\medskip

As we said above, the set consisting of continuous maps with infinite topological entropy is residual. On the other hand, it is easy to build non-autonomous dynamical systems with infinite topological entropy (take $\phi$ a continuous map with positive topological entropy, then $(\phi,\phi^{2},\phi^{2^{2}},\phi^{2^3},\dots)$ is a non-autonomous dynamical systems with infinite topological entropy).  This is the main reason to extend the concepts of mean dimension  and metric mean dimension to non-autonomous systems, since these become  a tool to classify non-autonomous dynamical systems with infinite topological entropy (see Theorem \ref{edee344}).



\medskip

 {In the next two sections we will extend  the mean dimension and the metric mean dimension for  a non-autonomous dynamical system $ \textit{\textbf{f}}=(f_{n})_{n=1}^{\infty}$, which will be denoted by 
$ \text{mdim}(X,\textit{\textbf{f}}\,)$. Furthermore, we will prove  some properties   which are valid for  the entropy of non-autonomous dynamical systems  (see \cite{K-S} and \cite{JeoCTE}).  An application of these properties is that, for any continuous maps $\phi$ and $\psi$ on $X$, the    compositions $\phi\circ \psi$ and $\psi\circ \phi$  have the same mean dimension (see Corollary \ref{desdede}). Furthermore, Remark \ref{dgrdgrw} proves the inequality $\text{mdim}_M(X,\phi^{p},d)\leq p\, \text{mdim}_M(X,\phi,d)$ can be strict.  Proposition    \ref{jsffeke}  proves  if  $X=[0,1]$ or $\mathbb{S}^{1},$ then for each $a\in [0,1]$, there exists a continuous map $\phi_{a}$ on $X$  with metric mean dimension equals to $a$.  In Theorem \ref{thm:non-wond} we show that, as the topological entropy,  the metric mean dimension is concentrated in the non-wandering set of the  dynamics.}

\medskip

In Section 5 we will discuss some upper bounds for the metric mean dimension of both autonomous and non-autonomous dynamical systems.


\medskip

As we said above,  the metric mean dimension for single continuous maps, and consequently for non-autonomous dynamical systems,  
depends on the metric $d$.  In Section \ref{section6} we will discuss some properties related to the invariance of the metric mean dimension under topological  equiconjugacy. 

\medskip

In the last section we will present some results related to the continuity of the metric mean dimension.

\medskip

Some  ideas given to proof the results  that  are well-known for the autonomous case work or can be adapted for the non-autonomous case. We will present these  proofs for the sake of  comprehensiveness.

\section{Mean dimension for non-autonomous dynamical systems}

Let $X$ be a compact metric space. In this section we will suppose that  $\textit{\textbf{f}}=(f_n)_{n=1}^{\infty}$  is a  non-autonomous
dynamical system, where $f_n:X\to X$
is a continuous map for all $n\geq 1$. We write $(X,\textit{\textbf{f}},d)$ to denote a non-autonomous dynamical system \textit{\textbf{f}} on $X$ endowed with the metric $d$. For $n,k\in \mathbb N$ define
$$
f_n^{(0)}:=I_X:= \text{the identity on }X\quad\quad \text{and} \quad \quad  f_n^{(k)}(x):=f_{n+k-1}\circ \dots \circ f_n(x)\quad\text{for }k\geq 1.
$$

 Set
\[ \mathcal{C}(X)=\{ (f_{n})_{n=1}^{\infty}: f_{n}:X\rightarrow X \text{ is a continuous map}\}.\]

\medskip

Given $\alpha $ an open cover of $X$ define
\[
\alpha_{0}^{n-1}=\alpha \vee f^{-1}_1(\alpha)\vee (f_1^{(2)})^{-1}(\alpha)\vee \dots \vee (f_1^{(n-1)})^{-1}(\alpha)
\]
and set 
\[\text{ord}(\alpha)=\sup_{x\in X}\sum_{U\in \alpha }1_{U}(x)-1\quad \quad \text{ and }\quad \quad \mathcal D(\alpha )
                    =\min_{\beta  \succ\alpha}\text{ord}(\beta),\]
                   where $1_{U}$ is the indicator function and  $\beta \succ \alpha$ means that $\beta$ is an open cover  of $X$ finner than $\alpha$.
\begin{definition} The \textit{mean dimension} of $\textit{\textbf{f}}\in\mathcal{C}(X)$ is defined to be
\begin{equation*} 
\text{mdim}(X,\textit{\textbf{f}}\,)=\sup_{\alpha}\lim_{n\to\infty}\frac{\mathcal D(\alpha_{0}^{n-1})}{n}.
\end{equation*}\end{definition}

By Corollary 2.5 of \cite{lind} we have that $\mathcal D(\alpha\vee \beta)\leq \mathcal D(\alpha)+\mathcal D(\beta)$, for
any open covers $\alpha$ and $\beta$. It follows that the limit that defines the mean dimension is well defined.

\begin{remark}\label{tete} We present 
  a list of   some important properties  about the mean dimension for both autonomous and non-autonomous dynamical systems: 
\begin{enumerate}
\item For a non-autonomous dynamical system given by the iterates of a single continuous map $f:X\to X$, i.e.,
  $\textit{\textbf{f}}=(f)_{n=1}^{\infty}$, the definition of mean dimension coincides with
  the one presented in \cite{lind}, that is,
   \(  \text{mdim}(X,(f)_{n=1}^{\infty})= \text{mdim}(X,f).\)
\item Recall that for a topological space $X$, the \textit{topological dimension} is defined as
$$
\text{dim}(X)=\sup_{\alpha}\mathcal{D}(\alpha)
$$
where $\alpha $ runs the open covers of $X$. If $\text{dim}(X)<\infty$, then $\mathcal{D}(\alpha_{0}^{n-1})\leq \text{dim}(X)$
for all $n\in \mathbb N$ and therefore $\text{mdim}(X,\textit{\textbf{f}}\,)=0$ for any   $\textit{\textbf{f}}\in  \mathcal{C}(X)$.
\item  In  \cite{lind}, Proposition 3.1, is proved that  $ 
\text{mdim}(X^{\mathbb{Z}},\sigma)\leq \text{dim}( X),
$
 where $\sigma$ is the shift on $ X^{\mathbb{Z}}$. Analogously we can prove  $ 
\text{mdim}(X^{\mathbb{N}},\sigma)\leq \text{dim}( X).
$
\item If $X=[0,1]^{m}$,  then $ 
\text{mdim}(X^{\mathbb{Z}},\sigma)= m
$ 
(see \cite{lind},  Proposition 3.3).
\item It is clear that if $Y\subseteq X$ is an invariant subset by a continuous map $\phi:X\rightarrow X$, then $\text{mdim}(Y,\phi)\leq \text{mdim}(X,\phi)$. 
We can define the mean dimension for any $Y\subseteq X$ as follows: let $\alpha$ be  an open cover  of
$X$ and consider $\alpha|_Y=\{U\cap Y:U\in\alpha\}$, the open cover of $Y$ given by the restriction of $\alpha$ to $Y$.
Then define
$$
\text{mdim}(Y,\textit{\textbf{f}}\,|_Y)=\sup_{\alpha}\lim_{n\to\infty}\frac{\mathcal D((\alpha|_Y)_{0}^{n-1})}{n}.
$$
It is clear   that $\text{mdim}(Y,\textit{\textbf{f}}\,|_Y)\leq \text{mdim}(X,\textit{\textbf{f}}\,)$. 
\item  A necessary condition for an invertible dynamical system  $\phi:X\rightarrow X$  to be imbeddable in $(([0, 1]^{m})^{\mathbb{Z}},\sigma )$ is that $\text{mdim} (X,\phi)\leq  m$ (see \cite{lind}, Corollary 3.4).
\item Any  nontrivial factor of $([0, 1]^{\mathbb{Z}}, \sigma)$ has positive mean dimension  (see \cite{lind}, Theorem 3.6).
\end{enumerate}
\end{remark}
 
We will show some properties of the mean dimension which are   valid for the topological entropy. 
Denote by $h_{top}(\textit{\textbf{f}}\,)$ the topological entropy of $\textit{\textbf{f}}$ (see \cite{K-S}, \cite{JeoCTE}).

\begin{definition}\label{composision} For any $p\geq 1$, set  $$ {\textit{\textbf{f}}^{\,(p)}=\{f_1^{(p)},f_{p+1}^{(p)},f_{2p+1}^{(p)},\dots\}=   \{ f_{p}\circ\cdots \circ f_{1}, f_{2p}\circ\cdots \circ  f_{p+1},f_{3p}\circ\cdots \circ  f_{2p+1},\dots \}}.$$\end{definition}

It is well-known that    $  
 h_{top}(\phi^{p})= p\, h_{top}(\phi)$ for any $p\geq1 $, where $\phi$ is any continuous map. For non-autonomous dynamical systems we have  $$
 h_{top}(\textit{\textbf{f}}^{\,(p)})\leq p\, h_{top}(\textit{\textbf{f}}\,)\quad \text{for any }p\geq 1
  $$ (see \cite{K-S}, Lemma 4.2).  In general, the equality $
 h_{top}(\textit{\textbf{f}}^{\,(p)})= p\, h_{top}(\textit{\textbf{f}}\,)$ is not valid, as we can see in the next example, which was given by Kolyada and Snoha in \cite{K-S}.  
 
 \begin{example} Take $\psi: [0,1]\rightarrow [0,1]$ defined by $ \psi(x)=1-|2x-1|$ for any $x\in [0,1]$. Consider $\textbf{\textit{f}}=(f_{n})_{n=1}^{\infty}$, where  
 $$
f_n(x)= \begin{cases}
    \psi^{(n+1)/2}(x), &  \text{ if }n \text{ is odd}, \\
     x/2^{n/2}, & \hbox{ if }n\text{ is even},
      \end{cases}
$$ for any $n\in\mathbb{N}$. 
Then $h_{top}(\textbf{\textit{f}}^{(2)})= 0 $ and   $h_{top}(\textbf{\textit{f}}) \geq \frac{\log 2}{2}$.
 \end{example}
 The equality $
 h_{top}(\textit{\textbf{f}}^{\,(p)})= p\, h_{top}(\textit{\textbf{f}}\,)$ is valid  if the sequence $\textit{\textbf{f}}=(f_{n})_{n=1}^{\infty}$ is equicontinuous (see \cite{K-S}, Lemma 4.4). On the other hand, the equality always holds for the mean dimension.

\begin{proposition}\label{pro:meandimension-periodic}
  For any  $\textit{\textbf{f}}=(f_{n})_{n=1}^{\infty}\in  \mathcal{C}(X)$ and $p\in \mathbb{N}$   we have
  $$
  \emph{mdim}(X,\textit{\textbf{f}}^{\,(p)})=p\, \emph{mdim}(X,\textit{\textbf{f}}\,).
  $$
\end{proposition}

\begin{proof}
Let  $\alpha$ be an open cover of $X$. Note that, for $k\in \mathbb N$,
\[
\lim_{k\to\infty}\frac{\mathcal D(\alpha\vee (f_1^{(p)})^{-1}(\alpha)\vee\dots \vee (f_1^{((k-1)p)})^{-1}(\alpha))}{k}\leq
\lim_{k\to\infty} \frac{\mathcal D(\alpha_0^{(k-1)p})}{k}\leq p  \lim_{k\to\infty} \frac{\mathcal D(\alpha_0^{(kp-1)})}{kp},
\]
which implies that  $\text{mdim}(X,\textit{\textbf{f}}^{\, (p)})\leq p\ \text{mdim}(X,\textit{\textbf{f}}\,)$. 
For the converse, note   that
\begin{align*}
  \alpha_{0}^{kp-1} & = \alpha_{0}^{p-1}\vee (f_1^{(p)})^{-1}(\alpha_{0}^{p-1})\vee (f_1^{(2p)})^{-1}(\alpha_{0}^{p-1})\vee\dots\vee (f_1^{((k-1)p)})^{-1}(\alpha_{0}^{p-1}) ,
\end{align*}
and therefore
\[
\text{mdim}(X,\textit{\textbf{f}}\,)=\sup_{\alpha}\lim_{k\to\infty}\frac{\mathcal D(\alpha_0^{(k-1)p})}{kp}\leq  \frac{\text{mdim}(X,\textit{\textbf{f}}^{\,(p)})}{p},
\]
which proves the proposition.
\end{proof}

  In  \cite{K-S}, Lemma 4.5, Kolyada and Snoha proved that $$ h_{top}(\sigma^{i}(\textit{\textbf{f}}\,))\leq h_{top}(\sigma^{j}(\textit{\textbf{f}}\,))\quad \text{for any }i\leq j,$$  where $\sigma$ is the left shift   $\sigma ((f_n)_{n=1}^{\infty})=(f_{n+1})_{n=1}^{\infty}$.  Furthermore, in \cite{JeoCTE},  Corollary 5.6, the author showed that if each $f_{n}$ is an homeomorphism then the equality holds, that is, the topological entropy for non-autonomous dynamical systems is independent on the   first maps on a sequence of homeomorphisms  $\textit{\textbf{f}}=(f_{n})_{n\in\mathbb{Z}}$.
Next proposition shows that these properties also hold for the mean dimension.

\begin{proposition}\label{pro:shift}
Let $i,j$ be two positive integers with $i\leq j$. Then
 $$ \emph{mdim}(X,\sigma^i(\textit{\textbf{f}}\,))\leq  \emph{mdim}(X,\sigma^j(\textit{\textbf{f}}\,)).$$ If each $f_{n}$ is a homeomorphism then the equality holds. 
\end{proposition}

\begin{proof}
 It is enough to prove the proposition for $i=0$ and $j=1$. For  any open cover $\alpha $ of $X$ we have
\begin{align*}
  \mathcal D(\alpha_{0}^{n-1})&\leq \mathcal D(\alpha) +\mathcal D(f_1^{-1}(\alpha\vee (f_2)^{-1}(\alpha)\vee (f_2^{(2)})^{-1}(\alpha)\vee\dots\vee (f_2^{(n-2)})^{-1}(\alpha)))\\
                              &=\mathcal D(\alpha)+\mathcal D(\alpha\vee (f_2)^{-1}(\alpha)\vee (f_2^{(2)})^{-1}(\alpha)\vee\dots \vee(f_2^{(n-2)})^{-1}(\alpha)).
\end{align*}
Thus
\begin{align*}
\lim_{n\to\infty}\frac{ \mathcal{D}(\alpha_{0}^{n-1})}{n}
                     &\leq \lim_{n\to\infty}\frac{ \mathcal{D}(\alpha)}{n}+ \lim_{n\to\infty}\frac{\mathcal D(\alpha\vee f_2^{-1}(\alpha)\vee (f_2^{(2)})^{-1}(\alpha)\vee\dots\vee (f_2^{(n-2)})^{-1}(\alpha))}{n}\\
                     &= \lim_{n\to\infty}\frac{n-1}{n}\frac{\mathcal D(\alpha\vee f_2^{-1}(\alpha)\vee (f_2^{(2)})^{-1}(\alpha)\vee\dots\vee (f_2^{(n-2)})^{-1}(\alpha))}{n-1}\\
                     &\leq\text{mdim}(X,\sigma(\textit{\textbf{f}}\,)),
\end{align*}
and therefore  $\text{mdim}(X,\textit{\textbf{f}}\,)\leq \text{mdim}(X,\sigma(\textit{\textbf{f}}\,)).$

Next, suppose that each $f_{n}$ is a homeomorphism.  Note that if $\beta $ refines $\alpha$ then $ \mathcal{D}(\beta)\geq  \mathcal{D}(\alpha)$.  Therefore, we have
\begin{align*}
 \mathcal D(\alpha\vee (f_2)^{-1}(\alpha)\vee (f_2^{(2)})^{-1}(\alpha)\vee\dots ) &=  \mathcal D(f_1^{-1}(\alpha\vee (f_2)^{-1}(\alpha)\vee (f_2^{(2)})^{-1}(\alpha)\vee\dots ))\\
                              &= \mathcal D((f_{1})^{-1}(\alpha)\vee (f_1^{(2)})^{-1}(\alpha)\vee (f_1^{(3)})^{-1}(\alpha)\vee\dots )\\
                              &\leq \mathcal D(\alpha \vee(f_{1})^{-1}(\alpha)\vee (f_1^{(2)})^{-1}(\alpha)\vee (f_1^{(3)})^{-1}(\alpha)\vee\dots ).
\end{align*}
Hence $\text{mdim}(X,\sigma(\textit{\textbf{f}}\,))\leq \text{mdim}(X, \textit{\textbf{f}}\,)$.
\end{proof}

If some $f_{n}$ is not a  homeomorphism, then the inequality above can be strict. In fact, take $f_{n}=f:X\rightarrow X$ for any $n\geq 2$, where $f$ is any continuous  map with positive mean dimension and $f_{1}:X\rightarrow X$ a constant map. Then  $\text{mdim}(X, \textit{\textbf{f}}\,)=0$ and $ \text{mdim}(X, \sigma(\textit{\textbf{f}}\,))= \text{mdim}(X, f)$. 

\medskip

Next corollary follows from Propositions \ref{pro:meandimension-periodic} and   \ref{pro:shift}:

\begin{corollary}\label{desdede}
  Let $\textit{\textbf{f}}=(f,g,f,g,\dots)$ and $\textit{\textbf{g}}=(g,f,g,f,\dots)$, where $f,g:X\to X$ are continuous maps.
  Then
  \[
   \emph{mdim}(X,\textit{\textbf{f}}\,)=\emph{mdim}(X,\textit{\textbf{g}}).
  \]
Therefore,  $$\emph{mdim}(X,f\circ g)=\emph{mdim}(X,g\circ f).$$\end{corollary}
\begin{proof}
It follows directly from Proposition \ref{pro:shift} that $\text{mdim}(X,\textit{\textbf{f}}\,)=\text{mdim}(X,\textit{\textbf{g}})$. Now, by Proposition \ref{pro:meandimension-periodic}   we have
\begin{align*}
  \text{mdim}(X,f\circ g)&=\text{mdim}(X,\textit{\textbf{f}}^{\,(2)})=2\,\text{mdim}(X,\textit{\textbf{f}}\,)=2\,\text{mdim}(X,\textit{\textbf{g}})\\
  &=\text{mdim}(X,\textit{\textbf{g}}^{\,(2)})= \text{mdim}(X,g\circ f),
  \end{align*}
  which proves the corollary.
\end{proof}

 It follows directly from Corollary \ref{desdede} that if  $f$ and $g$  are topologically conjugate continuous maps, then
  \[\text{mdim}(X,f)=\text{mdim}(X,g),\]
 since if $\phi$ is a topological conjugacy between $f$ and $g$, that is, $ \phi$ is a homeomorphism and $ \phi \circ f=g\circ \phi$, then
 \[\text{mdim}(X,f)=\text{mdim}(X,\phi^{-1}\circ \phi \circ f)=\text{mdim}(X, \phi \circ f\circ \phi^{-1})=\text{mdim}(X,g).\]

For any $\textit{\textbf{f}}=(f_n)_{n=1}^{\infty}\in \mathcal{C}(X)$, the \emph{asymptotic  mean dimension} is defined by the limit 
\begin{align*}\label{eq:assymp1}
\text{mdim}(X,\textit{\textbf{f}}\,)^*=\lim_{n\to\infty} \text{mdim}(X,\sigma^n(\textit{\textbf{f}}\,)).
\end{align*}
It follows from Proposition \ref{pro:shift} that the asymptotic  mean dimension always  exists. 

\begin{theorem}\label{prop:unif-limit}
 Let $\textit{\textbf{f}}=(f_n)_{n=1}^{\infty}\in  \mathcal{C}(X)$. If $\textit{\textbf{f}}$
  converges uniformly to a continuous map $f:X\to X$, then
\[
  \emph{mdim}(X,\textit{\textbf{f}}\,)^*\leq \emph{mdim}(X,f).
  \]
In particular, $\emph{mdim}(X,\textit{\textbf{f}}\,)\leq \emph{mdim}(X,f).$ 
\end{theorem}

\begin{proof}
  Let $(x_n)_{n\in\mathbb N}$ be a sequence of mutually different point converging to
  a point $x_0$. Define the map $F:\{x_n:n=0,1,\dots\}\times X\to\{x_n:n=0,1,\dots\}\times X$
  by $F:(x,y)\mapsto(\phi(x),\psi(x,y))$, where
\begin{align*}
  \phi(x_n)=\left\{
         \begin{array}{ll}
           x_0, & \hbox{ if }n=0 \\
           x_{n+1}, & \hbox{ if }n>0
         \end{array}\;\; \text{ and }\;\;
\psi(x_n,y)=\left\{
  \begin{array}{ll}
    f(y), & \hbox{ if }n=0 \\
    f_n(y), & \hbox{ if }n>0.
  \end{array}
\right.
       \right.
\end{align*}
Note that the non wandering set of $F$, $\Omega(F)$, is a subset  of the fix fiber $x_0\times X$. Since 
$$\text{mdim}(\{x_n:n=0,1,\dots\}\times X,F)=\text{mdim}(\Omega(F) ,F)$$ (by \cite[Lemma 7.2]{GTM}), we have that
\[
\text{mdim}(\{x_n:n=0,1,\dots\}\times X,F)= \text{mdim}(\{x_0\}\times X,F).
\]
 Therefore, $$\text{mdim}(\{x_m:m\geq k\}\times X,F)\leq \text{mdim}(\{x_0\}\times X,F)=\text{mdim}(\{x_n:n=0,1,\dots\}\times X,F), $$ for all $k>0$ (see Remark \ref{tete}, item (3)).
Next, note that by the definition of   $F$
 we have that $$\text{mdim}(\{x_m:m\geq k\}\times X,F)=\text{mdim}(X,\sigma^k(\textit{\textbf{f}}\,)),\quad \text{ for }k>0,$$ and
$\text{mdim}(\{x_0\}\times X,F)=\text{mdim}(X,f)$. Hence, $\text{mdim}( X,\sigma^k(\textit{\textbf{f}}\,))
\leq \text{mdim}(X,f)$, for all $k$.
\end{proof}

Next example proves that the inequality above can be strict. 
 
 \begin{example}\label{egre}
Let $\phi: I^{\mathbb{N}}\rightarrow I^{\mathbb{N}}$ be a continuous map with positive  mean dimension. For each $n\geq 1$, set $f_{n}:I^{\mathbb{N}}\times I^{\mathbb{N}}\rightarrow I^{\mathbb{N}}\times I^{\mathbb{N}}$ defined by $$f_{n}((x_{i})_{i\in\mathbb{N}},(y_{i})_{i\in\mathbb{N}})=((\lambda_{n} x_{i})_{i\in\mathbb{N}},(x_{i}(\phi(y))_{i})_{i\in\mathbb{N}} ),$$  where $\lambda _{n}\rightarrow 1$ and $\lambda_{n} \cdots \lambda_{1}\rightarrow 0$ as $n\rightarrow \infty$.
Note that $f_{n}$ converges uniformly on $I^{\mathbb{N}}\times I^{\mathbb{N}}$  to $f((x_{i})_{i\in\mathbb{N}},(y_{i})_{i\in\mathbb{N}})=((x_{i})_{i\in\mathbb{N}},(x_{i}(\phi(y))_{i})_{i\in\mathbb{N}})$ as $n\rightarrow \infty$ and \[ \text{mdim}(I^{\mathbb{N}}\times I^{\mathbb{N}},f)\geq \text{mdim}(\{(\dots,1,1,1,\dots)\}\times I^{\mathbb{N}},f)=\text{mdim}(I^{\mathbb{N}},\phi)>0. \]
On the other hand, note that $f_{k}^{n}(\bar{x},\bar{y})\rightarrow (\bar{0},\bar{0})$ as $n\rightarrow \infty$ for any $(\bar{x},\bar{y})\in I^{\mathbb{N}}\times I^{\mathbb{N}}$ and $k\geq 1$.  Hence $\text{mdim}(I^{\mathbb{N}}\times I^{\mathbb{N}},\sigma^{k}( \textit{\textbf{f}}\,))=0$
 for any $k\geq 1$, where $\textit{\textbf{f}}=(f_{n})_{n=1}^{\infty}$ and therefore  $\text{mdim}(I^{\mathbb{N}}\times I^{\mathbb{N}}, \textit{\textbf{f}}\,)^{\ast}=0$. \end{example}

\section{Metric mean dimension for non-autonomous dynamical systems}\label{section3}
Throughout this  section,   we will fix $\textit{\textbf{f}}=(f_n)_{n=1}^{\infty}\in  \mathcal{C}(X)$ where $X$ is a compact metric space with metric $d$.   For any  $n\in\mathbb{N}$ let   $d_n:X\times X\to [0,\infty)$ defined by
$$
d_n(x,y)=\max\{d(x,y),d(f_1(x),f_1(y)),\dots,d(f_1^{(n-1)}(x),f_1^{(n-1)}(y))\}.
$$ 
Thus $d_n$ is a metric on $X$ for all $n$ and generates the same topology induced by  $d$. Fix $\varepsilon>0$. We say that $A\subset X$ is an $(n,\textit{\textbf{f}},\varepsilon)$-\textit{separated} set
if $d_n(x,y)>\varepsilon$, for any two  distinct points  $x,y\in A$. We denote by $\text{sep}(n,\textit{\textbf{f}},\varepsilon)$ the maximal cardinality of an $(n,\textit{\textbf{f}},\varepsilon)$-separated
subset of $X$.  {Given an open cover $\alpha$ of $X$, we say that $\alpha$ is an 
$(n,\textit{\textbf{f}},\varepsilon)$-\textit{cover} if the $d_n$-diameter of any element of $\alpha$ is less than $\varepsilon$.} Let $\text{cov}(n,\textit{\textbf{f}},\varepsilon)$  be the minimum number of elements in an  $(n,\textit{\textbf{f}},\varepsilon)$-cover of $X$. We say that $E\subset X$ is an $(n,\textit{\textbf{f}},\varepsilon)$-\textit{spanning} set for $X$ if 
for any $x\in X$ there exists $y\in E$ such  that $d_n(x,y)<\varepsilon$. Let $\text{span}(n,\textit{\textbf{f}},\varepsilon)$ be the minimum cardinality
of any $(n,\textit{\textbf{f}},\varepsilon)$-spanning subset of $X$.   By the compactness of $X$, $\text{sep}(n,\textit{\textbf{f}},\varepsilon)$, $\text{span}(n,\textit{\textbf{f}},\varepsilon)$  and $\text{cov}(n,\textit{\textbf{f}},\varepsilon)$ are finite real numbers.

 \begin{definition}
  We define the \emph{lower  metric mean dimension} of $(X,\textit{\textbf{f}},d)$  and the \emph{upper metric mean dimension} of $(X,\textit{\textbf{f}},d)$ by
\begin{align*}
 \underline{\text{mdim}_M}(X,\textit{\textbf{f}},d)=\liminf_{\varepsilon\to0} \frac{\text{sep}(\textit{\textbf{f}},\varepsilon)}{|\log \varepsilon|}\quad \text{ and }\quad\overline{\text{mdim}_M}(X,\textit{\textbf{f}},d)=\limsup_{\varepsilon\to0} \frac{\text{sep}(\textit{\textbf{f}},\varepsilon)}{|\log \varepsilon|},
\end{align*}
respectively, where $\text{sep}(\textit{\textbf{f}},\varepsilon)=\underset{n\to\infty}\limsup \frac{1}{n}\log \text{sep}(n,\textit{\textbf{f}},\varepsilon)$. \end{definition}

It is not difficult to see that
$$
\underline{\text{mdim}_M}(X,\textit{\textbf{f}},d)=\liminf_{\varepsilon\to0} \frac{\text{span}(\textit{\textbf{f}},\varepsilon)}{|\log \varepsilon|}=\liminf_{\varepsilon\to0} \frac{\text{cov}(X,\varepsilon)}{|\log \varepsilon|},
$$
where $ \text{span}(\textit{\textbf{f}},\varepsilon)=\underset{n\to\infty}\limsup\frac{1}{n}\log \text{span}(n,\textit{\textbf{f}},\varepsilon)$ and  $\text{cov}(\textit{\textbf{f}},\varepsilon)=\underset{n\to\infty}\limsup\frac{1}{n}\log \text{cov}(n,\textit{\textbf{f}},\varepsilon).$
This fact holds for the upper metric mean dimension.
We will write $ {\text{mdim}_M}(X,\textit{\textbf{f}},d)$ to refer to both  $ \overline{\text{mdim}_M}(X,\textit{\textbf{f}},d)$  and $ \underline{\text{mdim}_M}(X,\textit{\textbf{f}},d)$. 
\medskip

Topological entropy for non-autonomous dynamical systems is invariant under uniform equi\-conjugacy (see \cite{K-S} and \cite{JeoCTE}). Metric mean dimension for single dynamical systems depends on the metric $d$ on $X$. Consequently, it is not an invariant under conjugacy and therefore it is not an invariant under uniformly equiconjugacy between   non-autonomous dynamical systems. Set $$\mathcal{B}=\{\rho: \rho \text{ is a metric on }X\text{ equivalent to }d \}$$ and take  
\begin{equation}\label{infmean}   {\text{mdim}_M}(X, \textit{\textbf{f}}\,)=\inf_{\rho\in \mathcal{B}}  {\text{mdim}_M}(X,\textit{\textbf{f}},\rho).\end{equation}

For single maps, ${\text{mdim}_M}(X, \phi)$ is an  invariant under topological conjugacy. In Proposition \ref{edee344} we will prove an analogous result for non-autonomous dynamical systems.

\begin{remark}\label{nmvnrit}
It follows from  the definition of the topological entropy for non-autonomous dynamical systems introduced in \cite{K-S} that 
if the topological entropy of the non-autonomous system $(X,\textit{\textbf{f}},d)$ is finite then its metric mean dimension is zero.
\end{remark}
 
 Next, we will present some examples of the  the metric mean dimension for  both  autonomous and non-autonomous dynamical systems. In Section \ref{Section5} we will show more examples. 
 
 \medskip

Take  $\mathbb{K}=\mathbb{N}$ or $ \mathbb{Z}$.  Consider  the metric $\tilde{d}$ on $X^{\mathbb{K}}$ defined by \begin{equation}\label{mnvc}\tilde{d}(\bar{x},\bar{y})= \sum_{i\in\mathbb{K}}\frac{1}{2^{|i|}}d(x_{i},y_{i}) \quad\text{ for }\bar{x}=(x_{i})_{i\in\mathbb{K}}, \bar{y}=(y_{i})_{i\in\mathbb{K}} \in X^{\mathbb{K}}.\end{equation} 
 
  Take $X=[0,1]$, endowed with the metric $d(x,y)=|x-y|$ for $x,y\in X.$ In \cite{lind3}, Example E, is proved that  $    \text{mdim}(X^{\mathbb{Z}}, \sigma, \tilde{d}) = 1.$
Analogously,   we can prove that    $   \text{mdim}(X^{\mathbb{N}}, \sigma, \tilde{d}) = 1:$

\begin{lemma}\label{bcbcbcbc}  Take $X=[0,1]$ endowed with the metric $d(x,y)=|x-y|$ for $x,y\in X.$ Thus  $$   \emph{mdim}(X^{\mathbb{N}}, \sigma, \tilde{d}) = 1.$$ \end{lemma}
\begin{proof}   Fix $\varepsilon>0$ and take $l = \lceil\log(4/\varepsilon)\rceil$, where $\lceil x\rceil =\min\{k\in \mathbb{Z}: x\leq k\}$. Note that $\sum_{n>l} 2^{-n}\leq \varepsilon/2$. Consider the open cover of $X$ given by 
$$I_{k}=\left(\frac{(k-1)\varepsilon}{12},\frac{(k+1)\varepsilon}{12} \right),\quad\text{for }0\leq k\leq \lfloor  12/\varepsilon\rfloor. $$
  Note that $I_{k}$ has length 
$\varepsilon/6$. Let $n\geq 1$. Next,  consider the following open cover  of $X^{\mathbb{N}}$: 
$$I_{k_{1}}\times I_{k_{2}}\times \cdots \times I_{k_{n+l}}\times X\times X\times\cdots,\quad  \text{ where }0\leq k_{1}, k_{2},\dots, k_{n+l}\leq \lfloor  12/\varepsilon\rfloor. $$ 
Each open set has diameter less than $\varepsilon$ 
 with respect to the
distance $\tilde{d}_{n}$ (see \eqref{mnvc}). Therefore  $$\text{cov}(n,\sigma,\varepsilon)\leq (1+ \lfloor  12/\varepsilon\rfloor)^{n+l}\leq(2+12/\varepsilon)^{n+1 +  12/\varepsilon}.$$ 
Hence 
$$\text{cov}(\sigma,\varepsilon)=\lim_{n\rightarrow \infty}\frac{\log\text{cov}(n,\sigma,\varepsilon)}{n}\leq \lim_{n\rightarrow \infty}\frac{(n+1 +  12/\varepsilon)\log (2+12/\varepsilon)}{n}=\log (2+12/\varepsilon).   $$ 
Thus $$ \text{mdim}(X^{\mathbb{N}}, \sigma, \tilde{d}) =\lim_{\varepsilon \rightarrow \infty} \frac{\text{cov}(\sigma,\varepsilon)}{|\log \varepsilon|}\leq 1. $$

On the other hand, any two distinct points in the sets 
$$ \{(x_{i})_{i\in\mathbb{N}}\in X^{\mathbb{N}} : x_{i} \in \{0,\varepsilon, 2\varepsilon, \dots ,\lfloor  1/\varepsilon\rfloor\varepsilon\} \text{ for all }0\leq i<n\} $$ 
have distance $\geq \varepsilon$ with respect to $d_{n}$. It follows that
$$\text{cov}(n,\sigma,\varepsilon)\geq (1+\lfloor  1/\varepsilon\rfloor)^{n}\geq (1/\varepsilon)^{n}.$$ Therefore
$$\text{cov}(\sigma,\varepsilon)\geq \lim_{n\rightarrow \infty}\frac{\log\text{cov}(n,\sigma,\varepsilon)}{n}=|\log \varepsilon| .$$ 
Hence  $\text{mdim}_{M}(X^{\mathbb{N}},\sigma,d)=1.$
\end{proof}

  {Next example proves that there exist dynamical systems on the interval with positive metric mean dimension (see also \cite{VV}). } 

\begin{example}\label{exfagner} Take $g:[0,1]\rightarrow [0,1]$, defined by $x\mapsto |1-|3x-1||$, and  $0=a_{0}<a_{1}< \cdots <a_{n}<\cdots$, where $a_{n}=\sum_{k=1}^{n}6/\pi ^{2}k^{2}$ for $n\geq 1$. For each $n\geq 1$, let 
 $T_{n}: J_{n}:=[a_{n-1},a_{n}] \rightarrow [0,1] $ be the unique increasing affine map from $J_{n}$ (which has length $6/\pi ^{2}n^{2}$)  onto $[0,1]$ and take any strictly increasing sequence of natural numbers $m_{n}$. 
 Consider the continuous map  $\phi:[0,1]\rightarrow [0,1]$ such that, for each $n\geq 1$, $\phi|_{J_{n}}= T_{n}^{-1}\circ g^{m_{n}}\circ T_{n}$.  
 \medskip
 
 Fix $n\geq 1$. Note that $J_{n}$ can be divided into $3^{m_{n}}$ intervals with the same length  $J_{n}(1),\dots,$ $  J_{n}({3^{m_{n}}})$, such that 
\[ \phi(J_{n}({i}))=J_{n} \quad\text{ for each } i\in\{1,\dots ,3^{m_{n}}\}.\] 

Next, $J_{n}({i})$ can be divided into $3^{n}$ intervals with the same length  $J_{n}({i},1),\dots,$ $  J_{n}({i},3^{n})$
such that $$\phi^{2}(J_{n}({i},s))=J_{n}\quad \text{for }i=1,\dots,3^{n}\quad\text{ and } s=1,\dots, 3^{n}.$$
Inductively, we can prove that for all $k\geq 1$ and    $(i_{1},\dots,i_{k})$, where  $i_{j}\in \{1,\dots, 3^{n}\}$,  we can divide  $J_{n}(i_{1},\dots,i_{k})$   into $3^{n}$ intervals with the same length $J_{n}(i_{1},\dots,i_{k},1),\dots,$ $  J_{n}(i_{1},\dots,i_{k},3^{n})$ such that 
$$\phi^{k+1}(J_{n}(i_{1},\dots,i_{k},i))=J_{n} \quad \text{for }i=1,\dots,3^{n}.$$
Each $J_{n}(i_{1},\dots,i_{k})$ has length $|J_{n}|/3^{kn}$ for each $k\geq 1$. Furthermore, each $J_{n}(i_{1},\dots,i_{k})$ has length $|J_{n}|/3^{km_{n}}$ for each $k\geq 1$.

\medskip

 Take $\varepsilon_n= |J_{n}|/3^{m_n}= 3/\pi ^{2}n^{2}3^{m_{n}}$ for each $n\geq 1$. If $x\in J_{n}(i_{1},\dots,i_{k})$ and $y\in J_{n}(j_{1},\dots,j_{k})$ where $(i_{1},\dots,i_{k})\neq (j_{1},\dots,j_{k})$ and each $i_{1},\dots,i_{k},j_{1},\dots,j_{k}$ is odd, then
 \[ d_{n+k}(x,y)\geq \varepsilon_{n}.\] 
 For each $k\geq1,$ there are more than $(3^{m_{n}}/2)^{k}$ intervals $J_{n}(i_{1},\dots,i_{k})$ with $i_{s}$ odd, $s=1,\dots,k$. 
 Hence    $\text{sep}(n+k,\phi, \varepsilon_{n})\geq (3^{m_{n}}/2)^{k}$ and then
$$
\text{sep}(\phi,\varepsilon_{n})\geq \lim_{k\to\infty}\frac{\log \text{sep}(n+k,\phi,\varepsilon_{n})}{k} \geq \log (3^{m_{n}}/2).$$ Therefore 
\begin{align*}
\overline{\text{mdim}_M}([0,1],\phi,| \cdot |)&\geq \lim_{n\to\infty}\frac{\log (3^{m_{n}}/2)}{-\log \varepsilon_{n}}= \lim_{n\to\infty}\frac{\log (3^{m_{n}}/2)}{-\log(3/\pi ^{2}n^{2}3^{m_{n}})}\\
&= \lim_{n\to\infty}\frac{\log (3^{m_{n}})+\log2}{\log(\pi ^{2}n^{2}/3)+\log (3^{m_{n}})}=1,\end{align*}
hence $ \overline{\text{mdim}_M}([0,1],\phi,|\cdot |)\geq1.$  We will obtain  from Proposition \ref{erfdy}  that $ \overline{\text{mdim}_M}([0,1],\phi,|\cdot |)\leq 1$. Therefore $  \overline{\text{mdim}_M}([0,1],\phi,|\cdot |)= 1.
$  \end{example}

Since $\phi(0)=0$ and $\phi(1)=1$, the map $\phi$ induces a continuous map on $\mathbb{S}^{1}$ with metric mean dimension equal to 1. More generally, we have:
  
 \begin{proposition}\label{jsffeke}  {Take $X=[0,1]$ or $\mathbb{S}^{1}.$ For each $a\in [0,1]$, there exists $\phi_{a}\in C^{0}(X)$ with $\emph{mdim}_{M}(\phi_{a})=a$.}\end{proposition}
 
\begin{proof} Any constant map has metric mean dimension equal to 0. On the other hand, Example \ref{exfagner} proves that there exist continuous maps on $X$ with metric mean dimension equal to 1.  Fix $a\in (0,1)$ and take  $r=\frac{1}{a}$.  Set $a_{0}=0$ and $a_{n}= \sum_{i=1}^{n}C(3^{-ir})$ for $n\geq 1$, where $C=1/ \sum_{i=1}^{\infty}3^{-ri}=1/(3^{r}-1)$. For each $n\geq 1$, take $J_{n}$,   
 $T_{n}$ and $g$ as in Example \ref{exfagner}.  
 Consider the continuous map  $\phi_{a}:[0,1]\rightarrow [0,1]$ such that, for each $n\geq 1$, $\phi_{a}|_{J_{n}}= T_{n}^{-1}\circ g^{n}\circ T_{n}$ (note that $\phi_{a}(0)=0$ and $\phi_{a}(1)=1$, consequently $\phi_{a}$ induces a continuous map on $\mathbb{S}^{1}$).    Fix $n\geq 1$. Each $J_{n}$ can be divided into $3^{n}$ intervals with the same length  $J_{n}(1),\dots,$ $  J_{n}({3^{n}})$, such that 
\[ \phi_{a}(J_{n}({i}))=J_{n} \quad\text{ for each } i\in\{1,\dots ,3^{n}\}.\] 
Next, $J_{n}({i})$ can be divided into $3^{n}$ intervals with the same length  $J_{n}({i},1),\dots,$ $  J_{n}({i},3^{n})$
such that $$\phi_{a}^{2}(J_{n}({i},s))=J_{n}\quad \text{for }i=1,\dots,3^{n}\quad\text{ and } s=1,\dots, 3^{n}.$$
Inductively, we can prove that for all $k\geq 1$ and    $(i_{1},\dots,i_{k})$, where  $i_{j}\in \{1,\dots, 3^{n}\}$,  we can divide  $J_{n}(i_{1},\dots,i_{k})$   into $3^{n}$ intervals with the same length $J_{n}(i_{1},\dots,i_{k},1),\dots,$ $  J_{n}(i_{1},\dots,i_{k},3^{n})$ such that 
$$\phi_{a}^{k+1}(J_{n}(i_{1},\dots,i_{k},i))=J_{n} \quad \text{for }i=1,\dots,3^{n}.$$
Each $J_{n}(i_{1},\dots,i_{k})$ has length $|J_{n}|/3^{kn}$ for each $k\geq 1$.

\medskip

Take $\varepsilon_n= |J_{n}|= C/3^{rn}$ for each $n\geq 1$. Each $  J_{n}(i_{1},\dots,i_{k})$ has $d_{k+1}$-diameter equal to  $\varepsilon_n$. Consequently,   $\text{cov}(k+1,\phi_{a}, \varepsilon_{n})\geq 3^{{n}k}$ and then $$ \text{cov}(\phi_{a},\varepsilon_{n})\geq \lim_{k\to\infty}\frac{\log \text{cov}(k+1,\phi_{a},\varepsilon_{n})}{k+1} \geq \log 3^{{n}}.$$     Therefore \begin{align*} {\text{mdim}_M}([0,1],\phi_{a},| \cdot |)&\geq \lim_{n\to\infty}\frac{\log 3^{n}}{-\log \varepsilon_{n}}= \lim_{n\to\infty}\frac{\log 3^{n}}{-\log(C/3^{nr})} = \lim_{n\to\infty}\frac{\log 3^{n}}{\log 3^{nr}}\\&=\lim_{n\to\infty}\frac{n\log 3}{nr\log3} = \frac{1}{r}=a.\end{align*}

On the other hand, fix $n\geq 1 $.  Let $m\geq n$ be such that $\sum_{i=m}^{\infty}C(3^{-ir})< \varepsilon_{n}$. Therefore \begin{equation}\label{bdhekn3rf}\text{cov}(\cup_{i=m}^{\infty}J_{i},k,\phi_{a}, \varepsilon_{n})=1\quad \text{ for any }k\geq 1.\end{equation}   Note that for each $k\geq 1$ and    $(i_{1},\dots,i_{k})$, where  $i_{j}\in \{1,\dots, 3^{n}\}$, the subintervals $J_{n}(i_{1},\dots,i_{k})$ have diameter less than $\varepsilon_n$ with the metric $d_{k}$ for any $k\geq 1$.  
Consequently, we have      \begin{equation}\label{bdheknrf}\text{cov}({J}_{n},k,\phi_{a}, \varepsilon_{n})\leq (3^{{n}})^{k} \quad \text{ for any }k\geq 1.\end{equation}
For each $i\in\{1,\dots,n-1\}$, divide each  interval $J_{i}$ into 
 $(3^{n})^{k+1}\lceil {|J_{i}|}/{|J_{n}|}\rceil$ subintervals with the same length, where $\lceil x\rceil =\min\{{j}\in \mathbb{Z}: x\leq {j}\}$. 
Each  subinterval  has $d_{k}$-diameter less than $\varepsilon_{n}$, thus \begin{equation}\label{bdheknrf2}\text{cov}(\cup_{i=1}^{n-1}{J}_{i},k,\phi_{a}, \varepsilon_{n})\leq \sum_{i=1}^{n-1} (3^{{n}})^{k+1}\lceil {|J_{i}|}/{|J_{n}|}\rceil.\end{equation}
For  $i\in\{n+1,\dots,m-1\}$, each $J_{i}$    has $d_{k}$-diameter less than $\varepsilon_{n}$, thus \begin{equation}\label{bdheknderf2}\text{cov}(\cup_{i=n+1}^{m-1}{J}_{i},k,\phi_{a}, \varepsilon_{n})\leq m-n-1\quad\text{ for any }k\geq 1.\end{equation}
By \eqref{bdhekn3rf}-\eqref{bdheknderf2}, we have 
\begin{align*}\text{cov}(\phi_{a}, \varepsilon_{n})&\leq \lim_{k\rightarrow\infty}\frac{\log[1+(3^{{n}})^{k}+\sum_{i=1}^{n-1} (3^{{n}})^{k+1}\lceil {|J_{i}|}/{|J_{n}|}\rceil+ m-n-1]}{k-1}\\
&\leq  \lim_{k\rightarrow\infty}\frac{\log[(\sum_{i=1}^{n-1} \lceil {|J_{i}|}/{|J_{n}|}\rceil+ m-n+1) (3^{{n}})^{k+1}]}{k-1}=\lim_{k\rightarrow\infty}\frac{\log (3^{n})^{k+1}}{k-1}\\
&=\log(3^{{n}}).\end{align*}
Hence  
\begin{align*}
 {\text{mdim}_M}([0,1],\phi_{a},| \cdot |)&\leq \lim_{n\to\infty}\frac{\log (3^{{n}})}{-\log \varepsilon_{n}}=  \lim_{n\to\infty}\frac{\log (3^{{n}})}{\log(3^{{rn}})} =\lim_{n\to\infty}\frac{n\log (3)}{(rn)\log(3)}  =a.\end{align*} 
 Therefore $   {\text{mdim}_M}([0,1],\phi_{a},|\cdot |)= a.$  \end{proof}

\begin{example}\label{lkjhfg}
Let $X=\{0,1\}^{\mathbb N}$ with its usual metric  and consider $\textit{\textbf{f}}=(f_{n})_{n=1}^{\infty}$, where $f_n:\{0,1\}^{\mathbb N}\to\{0,1\}^{\mathbb N}$ is given by $f_n(\omega)=\sigma^{2^n}(\omega)$, for any  $n\in\mathbb{N}$.
Note that $f_1^{(n)}(\omega)=\sigma^{2^{n+1}-2}(\omega)$. We claim that $\text{mdim}_{M}(X,\textit{\textbf{f}},d)=\infty$. Fix $\varepsilon>0$.  Take a  positive integer $k$
so that $2^{-(k+1)}\leq\varepsilon<2^{-k}$.
 Now consider $A\subset \{0,1\}^{\mathbb N}$ a $ (2^{n+1}-2,\varepsilon)$-separated set for the shift map $\sigma$ of maximum cardinality and note that $A$ is
an $(n,\varepsilon)$-separated set for $\textit{\textbf{f}}$.
Therefore,
$  \text{sep}(n,\textit{\textbf{f}},\varepsilon) \geq 2^{2^{n+1}-2+k}$ and then
\begin{align*}
 \frac{\log\text{sep}(n,\textit{\textbf{f}},\varepsilon) }{n\log \varepsilon}& \geq \frac{(2^{n+1}-2+k)\log2}{nk}.
\end{align*}
Hence, by the definition of the upper metric mean dimension, we have
$$
 {\text{mdim}_M}(X,\textit{\textbf{f}},d)=\limsup_{\varepsilon\to0}\limsup_{n\to\infty}\frac{\log\text{sep}(n,\textit{\textbf{f}},\varepsilon)}{n|\log\varepsilon|}=\infty.
$$
\end{example} 

In \cite{Zhu}, Zhu,   Liu,   Xu, and Zhang  showed that if $X$ is a  $k$-dimensional Riemannian manifold and $\textit{\textbf{f}}=(f_{n})_{n=1}^{\infty}$ is 
a sequence of $C^{1}$-maps on $X$   such that $a_{n}= \underset{x\in M}{\sup}\Vert D_{x}f_{n}\Vert<\infty$ for all $n\in \mathbb{N}$, then 
$$ h_{top}(\textit{\textbf{f}}\, ) \leq \max\left\{0,  \limsup_{n\rightarrow \infty}\frac{k}{n}\sum_{i=1}^{n-1}\log a_{i} \right\}. $$
Hence, by Remark \ref{nmvnrit}, we have: 
\begin{proposition} If $\limsup_{n\rightarrow \infty}\frac{k}{n}\sum_{i=1}^{n-1}\log a_{i} <\infty$, we have $ {\emph{mdim}_M}(M,\textit{\textbf{f}},d)=0.$\end{proposition}

  Any  sequence of homeomorphisms on both the interval or   the circle has zero topological entropy (see \cite{K-S}, Theorem D).   
Therefore, the metric mean dimension of any $\textbf{\textit{f}}$ on both the interval or   the circle   is equal to zero.  In the next example we will see that there exist non-autonomous dynamical systems consisting of diffeomorphisms on a surface with infinite metric mean dimension.

\begin{example}\label{hfkenrkflr} Let $\phi:\mathbb{T}^{2}\rightarrow \mathbb{T}^{2}$ be the diffeomorphism induced by a hyperbolic matrix $A$ with eigenvalue $\lambda>1$, where $\mathbb{T}^{2}$ is the torus endowed with the metric $d$ inherited from the plane.    Consider $\textit{\textbf{f}}=(f_{n})_{n=1}^{\infty}$ where    $f_{n}=\phi^{2^n} $ for each $i\geq 1$.  We have $|\text{Fix}(\phi^{n})|=\lambda^{n}+\lambda^{-n}-2, $ where  $\text{Fix}(\psi)$ is the set consisting of fixed points of a continuous map $\psi$ (see \cite{Katok}, Proposition 1.8.1). Furthermore,  $$\text{sep}(n,  \textit{\textbf{f}},1/4)\geq \text{sep}(2^n,  \phi,1/4)\geq \text{Fix}(\phi^{2^n})=\lambda^{2^n}+\lambda^{-2^n}-2$$
(see \cite{Katok}, 
Chapter 3, Section 2.e). Therefore, 
$$\lim_{n\rightarrow \infty}\frac{\text{sep}(n,  \textit{\textbf{f}},1/4)}{n}\geq \lim_{n\rightarrow \infty}\frac{\log \lambda^{2^n}}{n}=\infty,$$
and hence $\text{mdim}_{M}(\mathbb{T}^{2},\textit{\textbf{f}}, d)=\infty$.
\end{example}

  Suppose the   Hausdorff dimension of $X$ is finite. Let $\textit{\textbf{f}}=(f_{n})_{n=1}^{\infty}$ be a non-autonomous dynamical system where each $f_{n}$ is a $C^{r}$-map on $X$.   We have that if $h_{top}(\textit{\textbf{f}}\, )<\infty$ then $  {\text{mdim}_M}(X,\textit{\textbf{f}},d)=0.$ Therefore, if $\sup_{n\in\mathbb{N}}  L (f_{n})  <\infty$, where $L(f_{n})$ is the Lipschitz constant of $f_{n}$, we have that $h_{top}(\textit{\textbf{f}}\, )  <\infty$ and hence $   {\text{mdim}_M}(X,\textit{\textbf{f}},d)=0.$ Thus if   $\sup_{n\in\mathbb{N}}  L( f_{n})  <\infty$, then $   {\text{mdim}_M}(X,\textit{\textbf{f}},d)=0.$ 
In particular, if $X$ is a compact Riemannian manifold and $\textit{\textbf{f}}=(f_{n})_{n=1}^{\infty}$ is a sequence of differentiable maps  that  $\sup_{n\in\mathbb{N}}\Vert D f_{n}\Vert <\infty$, where $Df_{n}$ is the derivative of $f_{n},$ we have that $h_{top}(\textit{\textbf{f}}\, )  <\infty$ and hence $   {\text{mdim}_M}(X,\textit{\textbf{f}},d)=0.$ 
 
 \section{{Some fundamental properties of the metric mean dimension}}\label{section4}

In this section we show some properties which are well-known for topological entropy and metric mean dimension for  dynamical systems.  
 In the next proposition we will consider $\textit{\textbf{f}}^{\,(p)}$, which was defined in  Definition  \ref{composision}. 
 
 \medskip
 
 It is well-known  that  $h_{top}(\textit{\textbf{f}}^{\, (p)})\leq p\, h_{top}(\textit{\textbf{f}}\, )$ and if the sequence  $(f_{n})_{n=1}^{\infty}$ is equicontinuous, then the equality holds (see \cite{K-S}, Lemma 4.2).  For the case of the metric mean dimension, we always have that  $\text{mdim}_M(X,\textit{\textbf{f}}^{\,(p)},d)\leq p\, \text{mdim}_M(X,\textit{\textbf{f}},d)$.   However we will present an example where the inequality can be strict even for single continuous maps (see Remark \ref{dgrdgrw}). 
 
 \begin{proposition}\label{propo211}
For any $\textit{\textbf{f}}=(f_n)_{n=1}^{\infty}$ and  $p\in \mathbb N$, we have
  $$
  {\emph{mdim}_M}(X,\textit{\textbf{f}}^{\,(p)},d)\leq p\,  {\emph{mdim}_M}(X,\textit{\textbf{f}},d).
  $$
  Consequently (see \eqref{infmean}),  $$
   {\emph{mdim}_M}(X,\textit{\textbf{f}}^{\,(p)})\leq p\, {\emph{mdim}_M}(X,\textit{\textbf{f}}\,).
  $$
\end{proposition}
\begin{proof}
Note that, for any positive integer $m$, we have
$$\max_{0\leq j<m}d(f_1^{(jp)}(x),f_1^{(jp)}(y))\leq \max_{0\leq j<mp}d(f_1^{(j)}(x),f_1^{(j)}(y)).$$
Thus $\mbox{span}(m,\textit{\textbf{f}}^{\,(p)},\varepsilon)\leq \mbox{span}(mp,\textit{\textbf{f}},\varepsilon)$ and therefore 
$$
\text{span}(\textit{\textbf{f}}^{(p)},\varepsilon)=\limsup_{m\to \infty}\frac{1}{m}\log \mbox{span}(m, \textit{\textbf{f}}^{(p)},\varepsilon)
                                    \leq p\limsup_{m\to \infty}\frac{1}{mp}\log \mbox{span}(m ,\textit{\textbf{f}},\varepsilon)=p\ \text{span}(\textit{\textbf{f}},\varepsilon). $$ 
Hence ${\text{mdim}_M}(X,\textit{\textbf{f}}^{\,(p)},d)\leq p\,{\text{mdim}_M}(X,\textit{\textbf{f}},d).$
\end{proof}

\begin{remark}\label{dgrdgrw}
 In Example \ref{exfagner} we  prove that there exists a continuous map $\phi:[0,1]\rightarrow [0,1]$ such that  $ \underline{\text{mdim}_M}([0,1],\phi,d)=1$, where $d(x,y)=|x-y|$ for $x,y\in [0,1]$.  It follows from Proposition  \ref{erfdy} that for any $f:[0,1]\rightarrow [0,1]$  we have 
$\overline{\text{mdim}_{M}}([0,1],f,d)\leq1. $ Consequently, $\overline{\text{mdim}_{M}}([0,1],\phi^{n},d)\leq1$ for any $n\geq 1$, which proves that the inequality  in Proposition \ref{propo211} can be strict for autonomous systems and therefore for non-autonomous   systems.
\end{remark}

   If $A,B\subseteq X $ are invariant subsets  under a continuous map $\phi$, then   $$h_{top}(\phi)=\max \{ h_{top}(\phi|_{A})  ,h_{top}(\phi|_{B})\}.$$ 
It is clear this property is also valid for the metric mean dimension. 
\begin{proposition}\label{invariant}
  If $A,B\subseteq X $ are invariant subsets  under $\phi$, then   $$\emph{mdim}_{M}(X,\phi,d)=\max \{\emph{mdim}_{M}(X,\phi|_{A},d)  ,\emph{mdim}_{M}(X,\phi|_{B},d)\}.$$
\end{proposition}

If $A_{1}$, $A_{2}$, $\dots$ is a sequence of invariant subsets under $\phi$, then $$\max_{n\in\mathbb{N}} \{\text{mdim}_{M}(X,\phi|_{A_{n}},d) \}\leq \text{mdim}_{M}(X,\phi,d).$$ Example \ref{exfagner} proves that the inequality can be strict (the sets $J_{1}$, $J_{2}$,  $ \dots$ are invariant under $\phi$, however   $\text{mdim}_{M}(X,\phi|_{J_{n}},d)=0$ for each $n$).

 \medskip
  
  Metric mean dimension can be defined on any subset $A$ of $X$.   Kolyada  and  Snoha in \cite{K-S}, Lemma 4.1, proved that if $X= \cup_{i}^{n} A_{i}$, then  $$ h_{top}(\textit{\textbf{f}}\,)= \max_{i=1,\dots,n}h_{top}(\textit{\textbf{f}}\,|_{A_{i}}).$$
  Analogously we can prove that: 
  \begin{proposition}
  If $X= \cup_{i}^{n} A_{i}$, then  $$ \emph{mdim}_{M}(X, \textit{\textbf{f}}\, , d)= \max_{i=1,\dots,n}\emph{mdim}_{M}(X,\textit{\textbf{f}}\,|_{A_{i}},d).$$
  \end{proposition}

\begin{definition}
We say that $x\in X$ is a \textit{nonwandering point} for $\textit{\textbf{f}}$ if for every neighbourhood $U$ of $x$ there exist positive integers $k$ and $n$ with
$f_n^{(k)}(U)\cap U\not=\emptyset$. We denote by $\Omega({\textit{\textbf{f}}}\,)$ the set consisting of the nonwandering points of $\textit{\textbf{f}}$.
\end{definition} 

 It is well-known that for any continuous map $\phi:X\rightarrow X$ we have $h_{top}(\phi)=h_{top}(\phi|_{\Omega(\phi)}).$ This fact  was proved  for  non-autonomous dynamical systems by Kolyada and Snoha in  \cite{K-S}.  For mean dimension of single continuous maps this fact was proved by Gutman in \cite{GTM},  Lemma 7.2.   For the metric mean dimension of non-autonomous dynamical systems we also have:  

\begin{theorem}\label{thm:non-wond}
We have 
  \[
   {\emph{mdim}_M}(X,\textit{\textbf{f}},d)= {\emph{mdim}_M}(\Omega({\textit{\textbf{f}}\,)},\textit{\textbf{f}},d).
  \]
\end{theorem}

\begin{proof}
It is clear that $ {\text{mdim}_M}(X,\textit{\textbf{f}},d)\geq {\text{mdim}_M}(\Omega({\textit{\textbf{f}}}\,),\textit{\textbf{f}},d)$.
Fix $\varepsilon>0$ and $n\in\mathbb N$. Let $\alpha $ be an open $(n ,\textit{\textbf{f}},\varepsilon)$-cover of $X$ with minimum cardinality. Take $\beta$ a minimal finite open subcover of $\Omega(\textit{\textbf{f}}\,)$, chosen from $\alpha$ (note that $\beta$ is an $(n,\textit{\textbf{f}},\varepsilon)$-cover of $\Omega(\textit{\textbf{f}}\,)$). By the minimality of $\alpha$ we have that $\beta $ is an $(n,\textit{\textbf{f}}, \varepsilon)$-cover of $\Omega(\textit{\textbf{f}}\,)$  with minimum cardinality, which we denote by $ \text{cov}(\Omega(\textit{\textbf{f}}\,),n,\textit{\textbf{f}}, \varepsilon)$, i.e., $\text{Card}(\beta)=\text{cov}(\Omega(\textit{\textbf{f}}\,),n,\textit{\textbf{f}}, \varepsilon)$.

\medskip

The set $K=X\backslash \bigcup_{U\in\beta}U$ is compact and consists of wandering points. We can cover
$K$ by a finite number of wandering  subsets, each of them contained in some element of  $\alpha$. The sets defined before together
with $\beta$ form a finite open cover $\gamma(n)=\gamma$ of $X$, finer than $\alpha$. Consider, for each $k$, the open cover $\gamma(k,\textit{\textbf{f}}^{(n)})$
associated to the sequence $\textit{\textbf{f}}^{(n)}$. Note that each element of $\gamma(k,\textit{\textbf{f}}^{(n)})$ is of the form
$$
A_0\cap (f_1^{(n)})^{-1}(A_1)\cap (f_{1}^{(n)})\circ (f_{n+1}^{(n)})^{-1}(A_2)\cap\dots\cap(f_1^{(n)})^{-1}\circ \dots\circ(f_{(k-2)n+1}^{(n)})^{-1}(A_{k-1}),
$$
where $A_i\in\gamma$, for $i=0\dots,k-1$. It implies that $\gamma(k,\textit{\textbf{f}}^{(n)})$ is a $(k,\textit{\textbf{f}}^{(n)},\varepsilon)$-cover of $X$. Let $A_i$ and $A_j$ be  nonempty  open sets of $\gamma(k,\textit{\textbf{f}}^{(n)})$  for
some $i<j$. If $A_i=A_j$, then $$(f_{(j-1)n+1}^{(n)}\circ\dots\circ f_{in+1}^{(n)}) (A_i)=f_{in+1}^{(j-i)n}(A_i)$$ intersects $A_i=A_j$. In that case $A_i$ does not contain non-wandering points for $\textit{\textbf{f}}$ (and hence $A_i\in \beta$). Now we estimate the number of elements of $\gamma(k,\textit{\textbf{f}}^{(n)})$.
Setting \[{j:=\text{Card} \{A_{i}:i=0,1,\dots k-1\}\quad \text{and}\quad  m:=\text{Card}(\gamma(k,\textit{\textbf{f}}^{(n)})\backslash \beta)}, \] we have 
 $0\leq j\leq m$. 
In this case we have $\binom{m}{j}$ possibilities of the choice of a $j$-element subset of $\gamma(k,\textit{\textbf{f}}^{(n)})\backslash \beta$ and then these sets
can appear as various $A_i'$s in $k\cdot(k-1)\cdots(k-j+1)=k!\slash(k-j)!$ ways. For the rest of $A_i'$s we can choice any element of $\beta$.
So, the number of elements of $\gamma(k,\textit{\textbf{f}}^{(n)})$ is bounded by
$$
\sum_{j=0}^m\binom{m}{j}\frac{k!}{(k-j)!}\cdot (\text{Card}(\beta))^{k-j}.
$$
Since $k!\slash(k-j)!\leq k^m$ and $\binom{m}{j}\leq m!$, this number is not larger than $(m+1)\cdot m!\cdot k^m\cdot (\text{Card}(\beta))^{k}$.
Thus, using the fact that $\text{cov}(k,\textit{\textbf{f}}^{(n)},\varepsilon)\leq \text{Card}(\gamma(k,\textit{\textbf{f}}^{(n)}))$, we have
\begin{align*}
\limsup_{k\to\infty}\frac{1}{k}\log \text{cov}(k,\textit{\textbf{f}}^{(n)},\varepsilon)\leq \limsup_{k\to\infty}\frac{1}{k}\log(m+1)\cdot m!\cdot k^m\cdot (\text{Card}(\beta))^{k}=\log(\text{Card}(\beta)).
\end{align*}
As
\begin{align*}
\limsup_{k\to\infty}\frac{1}{k}\log \text{cov}(k,\textit{\textbf{f}}^{(n)},\varepsilon)=n  \limsup_{k\to\infty}\frac{1}{k}\log \text{cov}(k,\textit{\textbf{f}},\varepsilon),
\end{align*}
it follows that
\begin{align*}
\limsup_{k\to\infty}\frac{1}{k} \log \text{cov}(k,\textit{\textbf{f}},\varepsilon)
                                     &\leq\frac{1}{n} \log \text{cov}(\Omega(\textit{\textbf{f}}\,),n,\textit{\textbf{f}},\varepsilon).
\end{align*}
Taking the limsup as $n\to\infty$  we obtain $$\text{cov}(\textit{\textbf{f}},\varepsilon)\leq \limsup_{n\rightarrow\infty}\frac{1}{n} \log \text{cov}(\Omega(\textit{\textbf{f}}\,),n,\textit{\textbf{f}},\varepsilon):= \text{cov}(\Omega(\textit{\textbf{f}}\,),\textit{\textbf{f}},\varepsilon).$$ 
So,
\begin{align*}
 {\text{mdim}_M}(X,\textit{\textbf{f}},d)&=\liminf_{\varepsilon\to0}\frac{\text{cov}(\textit{\textbf{f}},\varepsilon)}{|\log\varepsilon|}
                                         \leq \liminf_{\varepsilon\to0}\frac{ \text{cov}(\Omega({\textit{\textbf{f}}}\,),\textit{\textbf{f}},\varepsilon)}{|\log\varepsilon|}
                                         = {\text{mdim}_M}(\Omega(\textit{\textbf{f}}\,),\textit{\textbf{f}},d),
\end{align*}
which proves  the theorem.
\end{proof}

\begin{definition}
A continuous map $\psi:X\to Y$ will be called $\alpha$-\textit{compatible} if it is possible to find a finite open
cover $\beta $ of $\psi(X)$ such that $\psi^{-1}(\beta)\succ\alpha$.
\end{definition}

Lindenstrauss  and  Weiss in \cite{lind}, Theorem 4.2, proved that for any metric $d$ compatible with the topology of $X$, we have \[
 \text{mdim}(X,\phi)\leq  \underline{\text{mdim}_M}(X,\phi,d) 
\] for any continuous map $\phi:X\rightarrow X$. These ideas work in order to show  the non-autonomous case: metric mean dimension is an upper bound for the mean dimension of non-autonomous dynamical systems. We will need the next proposition, whose proof  can be found in \cite{lind}, Proposition 2.4. 
\begin{proposition}\label{alpha-comp}
If $\alpha $ is an open cover of $X$, then $\mathcal D(\alpha)\leq k$ if and only if
there exists an $\alpha$-compatible continuous map $\psi:X\to K$, where $K$ has topological dimension $k$.
\end{proposition}

\begin{theorem}\label{efgetrr}
For any metric $d$ on $X$  compatible with the topology of $X$ we have that
\[
 \emph{mdim}(X,\textit{\textbf{f}}\,)\leq  \underline{\emph{mdim}_M}(X,\textit{\textbf{f}},d).
\]
\end{theorem}

\begin{proof}
Let $\alpha $ be an open cover of $X$. We can assume that $\alpha $ is of the form
\[\alpha=\{U_1,V_1\}\vee\dots \vee\{U_\ell,V_\ell\},\]
where each $\{U_i,V_i\}$ is  an open cover of $X$ with two elements. For each $1\leq i\leq \ell$ define
$\omega_i:X\to [0,1]$ by
\[
\omega_i(x)=\frac{d(x, X\backslash V_i)}{d(x, X\backslash U_i)+d(x, X\backslash V_i)}.
\]
It is not difficult to see that $\omega_i$ is Lipschitz, $U_i=\omega_i^{-1}([0,1))$ and $V_i=\omega_i^{-1}((0,1])$.

Let $C$ be a common bound for the Lipschitz constants of all $\omega_i$. For each positive integer
$N$ define $F(N,\cdot ):X\to[0,1]^{\ell N}$ by
\[
F(N,x)=(\omega_1(x),\dots, \omega_\ell (x),\omega_1(f_1(x)),\dots,\omega_\ell(f_1(x)),\dots,\omega_1(f_1^{(N)}(x)),\dots,\omega_\ell(f_1^{(N)}(x))).
\]
As $U_i=\omega_i^{-1}([0,1))$ and $V_i=\omega_i^{-1}((0,1])$ we have that  $F(N,\cdot)\succ\alpha_0^{N-1}$.

Now for each $S\subset \{1,\dots,\ell N\}$, for $x\in X$, denote by $F(N,x)_S$ the projection of $F(N,x)$ to the coordinates of
the index set $S$.

\noindent \textbf{Claim.} Let $\varepsilon>0$ and $D=\underline{\text{mdim}_M}(X,\textit{\textbf{f}},d)$. There exists $N(\varepsilon)>0$
so that, for all $N>N(\varepsilon)$ there exists $\xi\in (0,1)^{\ell N}$ which satisfies
$$
\xi_S\notin F(N,X)_S,
$$
for any subset $S\subset\{1,\dots,\ell N\}$ that satisfies $|S|>(D+\varepsilon)N$.
\begin{proof}
Let $\delta>0$ such that \[\delta<(2^\ell(2C)^{2D})^{-2\varepsilon} \quad\quad\text{ and }\quad \quad\frac{\text{sep}(\textit{\textbf{f}},\delta)}{\log \delta}=\underline{\text{mdim}_M}(X,\textit{\textbf{f}},d)+\frac{\varepsilon}{4}.
\]
We notice that for $N$ sufficiently large we can cover $X$ by $\delta^{-(D+\varepsilon\slash2)N}$ dynamical balls
$B(x,N,\delta)=\{y\in X:d_N(x,y)<\delta\}$. Since $C$ is the common Lipschitz constant for all $\omega_i$, 
we conclude that
\[
F(N,B(x,N,\delta))\subset\{a\in [0,1]^{\ell N}:\|F(N,x)-a\|_\infty<C\delta\},
\]
where $\|(a_1,\dots,a_{\ell N})-(b_1,\dots,b_{\ell N})\|_\infty=\sup_i|a_i-b_i|$. This fact implies that $F(N,X)$ can
be covered by $\delta^{-(D+\varepsilon\slash2)N}$ balls in the $\| \; \cdot\|_\infty$ norm of radius $C\delta$. Let
$B(1),\dots, B(K)$ be these balls, with $K=\delta^{-(D+\varepsilon\slash2)N}$.

Choose $\xi \in[0,1]^{\ell N}$ with uniform probability and notice that
\begin{align*}
  \mathbb P(\xi\in F(N,X)_S) & \leq\sum_{j=1}^{K}\mathbb{P}(\xi\in B(j)_S)\leq  \delta^{-(D+\varepsilon\slash2)N}(2C\delta)^{|S|},
\end{align*}
and so
\begin{align*}
  \mathbb P(\exists S:|S|>(D+\varepsilon)N \text{ and }\xi_S\in F(N,X)_S)
                & \leq \sum_{|S|>(D+\varepsilon)N} \mathbb{P}(\xi_S\in F(N,X)_S)\\
                & \leq (\sharp \text{ of such }S)\delta^{-(D+\varepsilon\slash2)N}(2C\delta)^{D+\varepsilon}N \\
                & \leq 2^{\ell N}((2C)^{2D}\delta^{\varepsilon\slash2})^N\ll1.
\end{align*}
Hence, with high probability, a random $\xi$ will satisfies the requirements.
\end{proof}

\noindent \textbf{Claim.} If $\pi: F(N,X)\to [0,1]^{\ell N}$ satisfies for both $a=0$ and $a=1$, and all $\xi\in[0,1]^{\ell N}$,
\[
    \{1\leq k\leq \ell N:\xi_k=a\}\subset \{1\leq k\leq \ell N:\pi(\xi)_S=a\},
\]
then $\pi\circ F(N,X)$ is compatible with $\alpha_0^{N-1}$.

\begin{proof}
  Given  $\xi\in [0,1]^{\ell N}$, define for $0\leq j< N$ and $1\leq i<\ell$
\[
W_{i,j}=\left\{
         \begin{array}{ll}
           (f_1^{(j)})^{-1}(U_i), & \hbox{ if }\xi_{j\ell+i}=0, \\
           (f_1^{(j)})^{-1}(V_i), & \hbox{ otherwise.}
         \end{array}
       \right.
\]
By the definition of $W_{i,j}$ we have that $(\pi\circ F(N,\cdot))^{-1}(\xi)\subset \displaystyle\bigcap_{1\leq i\leq \ell,\\ 0\leq j<N}W_{i,j}\in \alpha_0^{N-1}$. It follows that $\pi\circ F(N,X)$ is compatible with $\alpha_0^{N-1}$.
\end{proof}
For a fixed  $\varepsilon>0$, consider $\bar\xi$ and $N$ as in the first Claim. Set
\[
\Phi=\{\xi\in[0,1]^{\ell N}:\xi_k=\bar{\xi}_k \text{ for more than }(D+\varepsilon)N \text{ indexes }k\}.
\]
Then, $F(N,X)\subset \Phi^C=[0,1]^{\ell N}\backslash \Phi$.

Now, for each $m=1,2,\dots$, denote by $J_m$ the set
$$
J_m=\{\xi\in [0,1]^{\ell N}:\xi_i\in\{0,1\} \text{ for at least }m \text{ indexes }1\leq i\leq\ell N\}.
$$

Since $\bar\xi $ is in the interior of $[0,1]^{\ell N}$, one can define $\pi_1:[0,1]^{\ell N}\backslash \{\bar\xi\}\to J_1$
by mapping each $\xi$ to the intersection of the ray starting at $\bar\xi$ and passing through $\xi$ and $J_1$.
For each of the $(\ell N-1)$-dimensional cubes $I^t$ that comprises $J_1$ we can define a retraction on $I^t$ in a similar fashion
using as a center the projection of $\bar\xi$ on $I^t$. This will define a continuous retraction $\pi_2$ of $\Phi^C$ onto $J_2$.
As long as there is some intersection of $\Phi$ with the cubes in $J_m$ this process can be continued, thus we finally get
a continuous projection $\pi$ of $\Phi^C$ onto $J_{m_0}$, a space of topological dimension equals to $m_0$, with
$$
m_0\leq \lfloor D+\varepsilon\rfloor N+1,
$$
where $\lfloor x\rfloor =\max\{k\in \mathbb{Z}: k\leq x\}$. 
By construction,  $\pi$ satisfies the hypotheses of the second claim.
 Thus $\pi\circ F(N,\cdot)\succ\alpha_0^{N-1}$.
Moreover, since $F(N,X)\subset \Phi^C$, we have   $\pi(F(N,X))\subset J_{m_0}$.

Putting all together, we have constructed a $\alpha_0^{N-1}$ compatible function from $X$ to a space of
topological dimension less or equal to $\lfloor D+\varepsilon\rfloor N+1$, and so
$$
\frac{D(\alpha_0^{N-1})}{N}\leq \frac{\lfloor D+\varepsilon\rfloor N+1}{N}.
$$
As $\varepsilon$ goes to zero we get that $\text{mdim}(X,\textit{\textbf{f}}\,)\leq D$. 
\end{proof}

The inequality in the  theorem above can be strict for single maps and therefore for non-autonomous dynamical systems. 
 In \cite{lind2}, Theorem 4.3, is proved that if a continuous map $\phi:X\rightarrow X$ is an extension of a minimal system, then there is a metric $ d^{\prime}$ on $X$, equivalent to $d$, such that   $$\text{mdim}(X,\phi) = \underline{\text{mdim}_{M}}(X,\phi, d^{\prime}).$$


\section{Upper bound for the metric mean dimension}\label{Section5}
 
 As we saw in Remark \ref{tete}, we have  $ 
\text{mdim}(X^{\mathbb{K}},\sigma)\leq \text{dim}( X)
$, where $\mathbb{K}= \mathbb{Z}$ or $\mathbb{N}$. Furthermore,  if $X=I^{k}$, then $ 
\text{mdim}(X^{\mathbb{Z}},\sigma)= k$. 
 In this section we will prove that the metric mean dimension of the shift on  $X^{\mathbb{K}}$ is  equal to the box dimension of $X$ with respect to the metric $d$, which will be defined below. This fact implies that the metric mean dimension of any continuous map $\phi: X\rightarrow X$ is less or equal to the box dimension of $X$ with respect to the metric $d$ (see Proposition \ref{erfdy}).  

 \begin{definition} For $\varepsilon>0,$  let $N(\varepsilon)$ be the minimum number of closed balls of radious $\varepsilon$ needed
to cover $X$. The numbers 
$$\overline{\text{dim}_{B}}(X,d)=\limsup_{\varepsilon\rightarrow \infty}\frac{\log N(\varepsilon)}{|\log\varepsilon|} \quad \text{and}\quad \underline{\text{dim}_{B}}(X,d)=\liminf_{\varepsilon\rightarrow \infty}\frac{\log N(\varepsilon)}{|\log\varepsilon|}$$
 are called, respectively, the  \textit{upper Minkowski dimension} (or \textit{upper box
dimension}) of $X$ and the  \textit{lower Minkowski dimension} (or \textit{lower box
dimension}) of $X$, with respect to $d$. \end{definition}

For any metric space $(X,d)$ we have $$\text{dim}(X)\leq \text{dim}_{H}(X,d)\leq   \underline{\text{dim}_{B}}(X,d), $$
where $\text{dim}_{H}(X,d)$ is the Hausdorff dimension of $X$ with respect to $d$ (see \cite{Kawabata}, Section II, A). If $X=[0,1]$, then $\text{dim}(X)=\text{dim}_{H}(X,d)=   \underline{\text{dim}_{B}}(X,d)=1$. However, there exist sets such that the   inequalities above can be strict, as we will see in the next  example, which also proves  that neither $\text{dim}(X)$ nor  $\text{dim}_{H}(X,d)$ are    upper bounds for  $ \overline{\text{mdim}_{M}}(X^{\mathbb{Z}},\sigma,\tilde{d})$.

\begin{example}
Let  $A = \{0\} \cup \{1/n: n\geq 1\}$ endowed with the metric $d(x,y)=|x-y|$ for $x,y\in A$. In \cite{Kawabata}, Lemma 3.1, is proved that $\text{dim}_{H}(A) = 0$  while $\underline{\text{dim}_{B}}( A ) = 1/2.$  Furthermore, we have $$   \underline{\text{mdim}_{M}}(A^{\mathbb{Z}},\sigma,\tilde{d})= \underline{\text{dim}_{B}}( A ) = 1/2$$ (see \cite{lind3}, Section VII). 
\end{example}

Using the \textit{Classical Variational Principle}, in \cite{VV}, Theorem 5, the authors claim to have proven that for any $(X,d)$ 
 \begin{equation*} \overline{\text{mdim}_{M}}(X^{\mathbb{Z}},\sigma,\tilde{d})=\overline{\text{dim}_{B}}(X,d) .\end{equation*}

This fact can be proved generalizing the ideas given  in  \cite{lind3}, Example E:  
\begin{theorem}\label{bcbcbcbc1} For $\mathbb{K}=\mathbb{Z}$ or $\mathbb{N}$ we have    $$ \overline{\emph{mdim}_{M}}(X^{\mathbb{K}}, \sigma, \tilde{d}) =  \overline{\emph{dim}_{B}} (X,d) \quad \quad \text{and}\quad \quad \underline{\emph{mdim}_{M}}(X^{\mathbb{K}}, \sigma, \tilde{d})= \underline{\emph{dim}_{B}} (X,d). $$
\end{theorem}
\begin{proof} We will prove the case $\mathbb{K}=\mathbb{Z}$ (the case $\mathbb{K}=\mathbb{N}$ can be proved analogously as in Lemma \ref{bcbcbcbc}).   Fix $\varepsilon>0$ and take $l$ big enough such that   $\sum_{n>l} 2^{-n}\text{diam} (X)\leq \varepsilon/2 $. Let $m=N(\varepsilon)$ be the minimum   number of closed $\varepsilon$-balls $X_{1},\dots ,X_{m}$ needed
to cover $X$. Consider the open cover of $X^{\mathbb{Z}}$ given by the open sets
$$\cdots \times X \times X_{k_{-l}}\times X_{k_{-l+1}}\times \cdots \times X_{k_{n+l}}\times X\times\cdots,\quad  \text{ where }1\leq k_{-l}, k_{-l+1},\dots, k_{n+l}\leq    m. $$ 
Note that each one of these open sets has diameter less than $4\varepsilon$ 
 with respect to the
distance $\tilde{d}_{n}$ on $X^{\mathbb{Z}}$. Therefore  $ \text{cov}(n,\sigma,4\varepsilon)\leq m ^{n+2l+1}$ and hence 
$$\text{cov}(\sigma,4\varepsilon)=\lim_{n\rightarrow \infty}\frac{\log\text{cov}(n,\sigma,4\varepsilon)}{n}\leq \lim_{n\rightarrow \infty}\frac{(n+2l+1)\log (m)}{n}=\log N(\varepsilon),  $$ 
which implies that $$ \overline{\text{mdim}_{M}}(X^{\mathbb{Z}}, \sigma, \tilde{d}) =\limsup_{\varepsilon \rightarrow \infty} \frac{\text{cov}(\sigma,4\varepsilon)}{|\log 4\varepsilon|}\leq \limsup_{\varepsilon \rightarrow \infty} \frac{\log N(\varepsilon)}{|\log 4\varepsilon|}=  \limsup_{\varepsilon \rightarrow \infty} \frac{\log N(\varepsilon)}{|\log4 + \log \varepsilon|}=\overline{\text{dim}_{B}} (X,d) $$ and
$$ \underline{\text{mdim}_{M}}(X^{\mathbb{Z}}, \sigma, \tilde{d}) =\liminf_{\varepsilon \rightarrow \infty} \frac{\text{cov}(\sigma,4\varepsilon)}{|\log 4\varepsilon|}\leq  \underline{\text{dim}_{B}} (X,d), $$

To prove the converse inequality, for $\varepsilon>0$  let  $\{x_1,x_2,\dots,x_{N(\varepsilon)}\}$ be a maximal set of  points in $X$ which are $\varepsilon$-separated.
 For $n\geq 1$, consider the set
$$ \{(y_{i})_{i\in\mathbb{Z}}\in X^{\mathbb{Z}} : y_{i} \in \{x_1,x_2,  \dots ,x_{N(\varepsilon)}\} \text{ for all }-l\leq i\leq n+l\} $$ 
and notice that it is $(\sigma,n,\varepsilon)$-separated and its cardinality is bounded from below by $N(\varepsilon)^{n+2l+1}$. So 
$$
\text{sep}(\sigma,\varepsilon)\geq \lim_{n\to\infty}\frac{\log N(\varepsilon)^{n+2l+1}}{n}=\log N(\varepsilon),
$$
and it implies that $$\underline{\text{mdim}_{M}}(X^{\mathbb{Z}}, \sigma, \tilde{d}) \geq \underline{\text{dim}_{B}} (X,d),$$
which proves the theorem. 
\end{proof}

 Next proposition proves  the metric mean dimension of any dynamical system is bounded by the box dimension of the space (see \cite{VV}, Remark 4).

\begin{proposition}\label{erfdy}  For any continuous map $\phi:X\rightarrow X$  we have 
$$\overline{\emph{mdim}_{M}}(X,\phi,d) \leq \overline{\emph{dim}_{B}} (X,d) \quad \text{ and }\quad \underline{\emph{mdim}_{M}}(X,\phi,d) \leq \underline{\emph{dim}_{B}} (X,d) .$$  In particular, if $X=[0,1]$, then  $$\underline{\emph{mdim}_{M}}(X,\phi,d)\leq \overline{\emph{mdim}_{M}}(X,\phi,d)\leq 1.$$ 
\end{proposition}
\begin{proof}
Consider the embedding  $\psi: X\rightarrow X^{\mathbb{N}}$, defined by  $x\mapsto\psi(x)= (x,\phi(x),\phi^{2}(x),\dots)$. We have $\sigma\circ \psi=\psi \circ \phi$.   Therefore, $Y=\psi(X)$ is a closed subset of  $ X^{\mathbb{N}}$ invariant by $\sigma$.  Take the metric $d_{\psi}$ on $X$ defined by $ d_{\psi}(x,y)= \tilde{d}(\psi(x),\psi(y)),$  for any  $x,y\in X.$  Clearly $d(x,y)\leq d_{\psi}(x,y)$ for any $x,y\in X$, therefore any $ (n,\phi,\varepsilon)$-separated subset of $X$ with respect to $d$ is a $ (n,\phi,\varepsilon)$-separated subset of $X$ with respect to $d_{\psi}$. Hence $$\overline{\text{mdim}_{M}}(X,\phi,d)\leq  \overline{\text{mdim}_{M}}(X,\phi,d_{\psi}) =\overline{\text{mdim}_{M}}(Y,\sigma|_{Y},\tilde{d})\leq \overline{\text{mdim}_{M}}(X^{\mathbb{N}},\sigma,\tilde{d}) \leq \overline{\text{dim}_{B}} (X,d) $$
and, analogously, $ \underline{\text{mdim}_{M}}(X,\phi,d)\leq   \underline{\text{dim}_{B}} (X,d). $
\end{proof}
 
 Example \ref{exfagner} proves that there exist dynamical systems $\phi:X\rightarrow X$ such that  $$\overline{\text{mdim}_M}(X,\phi,d)=\overline{\text{dim}_{B}} (X,d)\quad \text{ and }\quad \underline{\text{mdim}_M}(X,\phi,d)=\underline{\text{dim}_{B}} (X,d).$$

  \medskip

 We can consider the \emph{asymptotic metric mean dimension} as the   limit
\begin{align*}\label{eq:assymp}
 {\text{mdim}_M}(X,\textit{\textbf{f}},d)^*=\limsup_{i\to\infty}{\text{mdim}_M}(X,\sigma^i(\textit{\textbf{f}}\,),d).
\end{align*}

\begin{theorem}\label{maintheoremE} If $\textit{\textbf{f}}=(f_n)_{n=1}^{\infty}$   converges uniformly to a continuous map $f:X\to X$, then, for any $k\geq 1$, 
\begin{equation}\label{tergss}  {\emph{mdim}_M}(X,\sigma^{k}(\textit{\textbf{f}}\,),d)\leq  {\emph{mdim}_M}(X,f,d).
  \end{equation}
  Consequently,  \[ {\emph{mdim}_M}(X,\textit{\textbf{f}},d)^*\leq  {\emph{mdim}_M}(X,f,d). \]
\end{theorem}
\begin{proof}
See the proof of  Theorem \ref{prop:unif-limit} and use Theorem \ref{thm:non-wond}. 
\end{proof}

We can prove, as in Example \ref{egre}, that the  inequality above  can be strict.

\medskip

 Theorem \ref{maintheoremE} and Proposition \ref{erfdy} imply  that: 
\begin{corollary}\label{efrrttt}  
 If  $\textit{\textbf{f}}=(f_{n})_{n=1}^{\infty}$
  converges uniformly to a continuous map on $X$, then \begin{equation*}  \overline{\emph{mdim}_M}(X, \textit{\textbf{f}},d)\leq\overline{\emph{dim}_{B}} (X,d)\quad\text{and}\quad \underline{\emph{mdim}_M}(X, \textit{\textbf{f}},d)\leq\underline{\emph{dim}_{B}} (X,d).
  \end{equation*}  and therefore  \[\overline{\emph{mdim}_M}(X,\textit{\textbf{f}},d)^*\leq\overline{\emph{dim}_{B}} (X,d) \quad \text{and}\quad \underline{\emph{mdim}_M}(X,\textit{\textbf{f}},d)^*\leq\underline{\emph{dim}_{B}} (X,d) . \]  In particular, if $X=[0,1],$ then $\overline{\emph{mdim}_M}(X,\textit{\textbf{f}},d)^*\leq 1$. \end{corollary}

  Example \ref{lkjhfg}  proves that the box dimension is not an upper bound for the metric mean dimension of  sequences that are not convergent. Next example shows the inequality  in  Corollary \ref{efrrttt}   can be strict.

\begin{example} For each $n\geq 1,$ take $m_{n}=n$ and
$$
f_n(x)= \begin{cases}
    \phi(x), &  \text{ if }x\in[0,a_{n+1}], \\
    a_{n+1}, & \hbox{ if }x\in [a_{n+1},1],
      \end{cases}
$$ where $ \phi$ is the map  in Example \ref{exfagner}.  
Thus $f_n$ converges uniformly to $\phi$ as $n\rightarrow \infty$. In \cite{K-S}, Figure 3, is proved that the topological entropy $h_{top}((f_{n+k})_{n=1}^{\infty})=k\log 3$ for each $k\geq1$. Hence, $ \overline{\text{mdim}_M}([0,1],(f_{n+k})_{n=1}^{\infty},|\cdot |)=0$ and therefore $$ \overline{\text{mdim}_M}([0,1],(f_{n})_{n=1}^{\infty},|\cdot |)^{\ast}=0< \overline{\text{mdim}_M}([0,1],\phi,|\cdot |)=1.$$ 
\end{example}

 \begin{example}  
The sequence
$$
g_n(x)= \begin{cases}
    \phi(x), &  \text{ if }x\in[0,a_{n+1}], \\
    x, & \hbox{ if }x\in [a_{n+1},1].
      \end{cases}
$$
  converges uniformly to $\phi$ as $n\rightarrow \infty$, where $ \phi$ is the map  in Example \ref{exfagner}. Note that $g_{1}^{(n+k)}|J_{n}=\phi^{k}|_{J_{n}},$ for $n\geq 1,k\geq1$ (see Example \ref{exfagner}).  Hence $$\text{sep}(2n+k,(g_{i})_{i=1}^{\infty}, \varepsilon_{n})\geq (3^{m_{n}}/2)^{k},\quad\text{ and then }\quad \text{sep}((g_{i})_{i=1}^{\infty},\varepsilon_{n}) \geq \log (3^{m_{n}}/2).$$ Therefore $
\overline{\text{mdim}_M}([0,1],(g_{i})_{i=1}^{\infty},| \cdot |)\geq 1.$ By \eqref{tergss} we obtain that   $ \overline{\text{mdim}_M}([0,1],(g_{i})_{i=1}^{\infty},|\cdot |)= 1$.   Note that $ \overline{\text{mdim}_M}([0,1],g_{i},|\cdot |)= 0$ for any $i\geq1$. \end{example}

\section{Uniform equiconjugacy and metric mean dimension}\label{section6}

   We say that  the systems  $\textit{\textbf{f}}=(f_{n})_{n=1}^{\infty}$ on $(X,d)$ and $\textit{\textbf{g}}=(g_{n})_{n=1}^{\infty}$ on $(Y,d^{\prime})$
are \textit{uniformly equiconjugate} if there exists a equicontinuous sequence of  homeomorphisms $h_n: X\to Y$ so that
$h_{n+1}\circ f_n=g_n\circ h_n$, for all $n\in \mathbb N$,  that is, the following diagram
\[ \begin{CD}
    X @>f_1>> X@>f_2>>\dots @>f_n>>X\\
    @VVh_{1} V      @VV h_{2} V      @.     @VVh_{n+1}V\\
   Y @>g_1>> Y@>g_2 >>\dots @>g_n>>Y
  \end{CD}
\]
is commutative for  all $n\in \mathbb N$. 
In the case where $h_n=h$, for all $n\in\mathbb N$, we say that
$ \textit{\textbf{f}}$ and $\textit{\textbf{g}}$ are \textit{uniformly conjugate}.

\medskip

 Note that the notion of uniform equiconjugacy does not depend on the metric on $X$ and $Y$. Indeed,  if $d^{\ast}$   and $d^{\star}$ are another  metrics on $X$ and $Y$, respectively, then $(X,\textit{\textbf{f}}, d)$ and $(X,\textit{\textbf{f}}, d^{\ast})$ are uniformly equiconjugate by the sequence $(I_{X})_{n=1}^{\infty}$   and  $(Y,\textit{\textbf{g}}, d')$ and $(Y,\textit{\textbf{g}}, d^{\star})$ are uniformly equiconjugate by the sequence $(I_{Y})_{n=1}^{\infty}$. Hence, if  $(X,\textit{\textbf{f}}, d)$ and $(Y,\textit{\textbf{g}}, d')$ are uniformly equiconjugate by the sequence $(h_{n})_{n=1}^{\infty}$, then $(X,\textit{\textbf{f}}, d^{\ast})$ and $(Y,\textit{\textbf{g}}, d^{\star})$ are uniformly equiconjugate by the sequence $(I_{Y}\circ h_{n}\circ I_{X})_{n=1}^{\infty}$.

\begin{theorem}\label{edee344}
 Let $\textit{\textbf{f}}=(f_{n})_{n=1}^{\infty}$ and $\textit{\textbf{g}}=(g_{n})_{n=1}^{\infty}$ be two non-autonomous dynamical systems defined on the  metric spaces $(X,d)$ and $(Y,d^{\prime})$ respectively.
\begin{enumerate}[(i)]
\item  If $\textit{\textbf{f}}$ and $\textit{\textbf{g}}$ are uniformly conjugate then
\begin{align*}
\emph{mdim}(X,\textit{\textbf{f}}\,)&=\emph{mdim}(X,\textit{\textbf{g}}).
\end{align*}
\item If $(X,\textit{\textbf{f}}\,)$ and $(Y,\textit{\textbf{g}})$ are uniformly equiconjugate by a sequence of homeomorphisms $(h_n)_{n=1}^{\infty}$ that satisfies $\inf_n \{d(h^{-1}_n(y_1),h_n^{-1}(y_2))\}>0$  for any   $y_1,y_2\in Y$, then (see \eqref{infmean})
\begin{equation*}
 {\emph{mdim}_M}(X,\textit{\textbf{f}}\,)\geq  {\emph{mdim}_M}(Y,\textit{\textbf{g}}).
\end{equation*}
\item  If $(X,\textit{\textbf{f}}\,)$ and $(Y,\textit{\textbf{g}})$ are uniformly equiconjugate by a sequence of homeomorphisms $(h_n)_{n=1}^{\infty}$ that satisfies $\inf_n \{ d^{\prime}(h_n(x_1),h_n(x_2))\}>0$  for any $x_1,x_2\in X$, then
\begin{equation*}
 {\emph{mdim}_M}(X,\textit{\textbf{f}}\,)\leq  {\emph{mdim}_M}(Y,\textit{\textbf{g}}).
\end{equation*}
\item  If $(X,\textit{\textbf{f}}\,)$ and $(Y,\textit{\textbf{g}})$ are uniformly equiconjugate by a sequence of homeomorphisms $(h_n)_{n=1}^{\infty}$ that satisfies $\inf_n \{d(h^{-1}_n(y_1),h_n^{-1}(y_2)), d^{\prime}(h_n(x_1),h_n(x_2))\}>0$  for any $y_{1},y_{2}\in Y$ and  $x_1,x_2\in X$, then
\begin{equation*}
 {\emph{mdim}_M}(X,\textit{\textbf{f}}\,)=  {\emph{mdim}_M}(Y,\textit{\textbf{g}}).
\end{equation*}
\end{enumerate}
\end{theorem}
\begin{proof}
 (i) Let $h: X\to Y$ be a homeomorphism which conjugates  $\textit{\textbf{f}}$ and $\textit{\textbf{g}}$, i.e.,
$h\circ f_1^{(n)}=g_1^{(n)}\circ h$ for all $n\in \mathbb N$. For an open cover $\alpha$
of $X$, consider $\beta=h(\alpha)$, which is an  open cover of $Y$. Now we notice that
\begin{align*}
  \beta_0^{n-1} & =h(\alpha) \vee g_1^{-1 }(h(\alpha))\vee\dots\vee (g_1^{(n-1)})^{-1}(h(\alpha)) \\
                & =h(\alpha)\vee (h\circ f_1^{-1}\circ h^{-1})(h(\alpha))\vee \dots \vee (h\circ (f_1^{(n-1)})^{-1}\circ h^{-1})(h(\alpha))\\
                & =h(\alpha_0^{n-1}).
\end{align*}
It implies that $\mathcal{D}(h(\alpha_0^{n-1}))=\mathcal{D}(\alpha_0^{n-1})$. Since, for any open cover $\beta$ of $Y$
is of the form $h(\alpha)$, for some open cover $\alpha$ of $X$,
\[\text{mdim}(X,\textit{\textbf{f}}\,)=\sup_{\alpha}\lim_{n\to\infty}\frac{\mathcal{D}(\alpha_0^{n-1})}{n}=\sup_{\beta}\lim_{n\to\infty}\frac{\mathcal{D}(\beta_0^{n-1})}{n}=\text{mdim}(Y,\textit{\textbf{g}}).
\]

\noindent (ii) 
Let $(h_n)_{n=1}^{\infty}$ be the sequence of equicontinuous homeomorphisms that equiconjugates $\textit{\textbf{f}}$ and $\textit{\textbf{g}}$. So,
$$f_n\circ\dots\circ f_1=h_{n+1}^{-1}\circ g_n\circ\dots\circ g_1\circ h_1.$$
By assumption we have
\begin{equation*}
 \inf_n \{d(h^{-1}_n(y_1),h_n^{-1}(y_2))\}>0, \text{ for any } y_1\not=y_2\in Y.
\end{equation*}
 Hence, we can define on $Y$   the metric 
\begin{align*}
d^\star(y_1,y_2):=\inf_n\{ d(h_{n}^{-1}(y_1),h_{n}^{-1}(y_2))\}.
\end{align*}
In particular, if $S\subset X$ is a $(m,\textit{\textbf{f}} ,\varepsilon)$-spanning set of $X$ in the metric $d$ and $x_1,x_2\in S$, then
\begin{align*}
d_m^\star(h_1(x_1),h_1(x_2))&=\max\{d^{\star}(h_1(x_1),h_1(x_2)),\dots,d^{\star}
(g_1^{m-1}(h_1(x_1)),g_1^{m-1}(h_1(x_2)))\} \\
     &\leq \max\{d(x_1,x_2), d(h_2^{-1}(g_1(h_1(x_1))),h_2^{-1}(g_1(h_1(x_2)))),\\
            &\quad \quad \dots,d(h_{m+1}^{-1}(g_1^{m-1}(h_1(x_1))),h_{m+1}^{-1}(g_1^{m-1}(h_1(x_2))))\}\\
                            &=d_m(x_1,x_2)\leq \varepsilon.
\end{align*}
It follows that $h_1(S)$ is an $(m, \textit{\textbf{g}} ,\varepsilon)$-spanning set of $Y$ in the metric $d^{\star}$. So we obtain that
$$ 
 {\text{mdim}_M}(X,\textit{\textbf{f}}, d)\geq   {\text{mdim}_M}(Y,\textit{\textbf{g}}, d^{\star}) , 
$$ and therefore $
 {\text{mdim}_M}(X,\textit{\textbf{f}}\,)\geq   {\text{mdim}_M}(Y,\textit{\textbf{g}}). 
$

By an analogous argument we can prove (iii).  Item (iv) follows from (ii) and (iii). 
\end{proof}

 Clearly the theorem implies that if  $\phi:X\rightarrow X$ and $\psi: X\rightarrow X$ are topologically conjugate continuous maps, then      $$
 {\text{mdim}_M}(X,\phi)=  {\text{mdim}_M}(X,\psi), 
$$ which is a well-known fact.


\medskip

The next corollaries  follow from  Theorem \ref{edee344}.

\begin{corollary}\label{newe} If $f_{1},\dots,f_{i},g_{1},\dots,g_{i}$ are homeomorphisms, $ \textit{\textbf{f}}=(f_{1},  \dots, f_{i}, f_{i+1}, f_{i+2}, \dots)$ and $ \textit{\textbf{g}}=(g_{1},\dots, g_{i}, f_{i+1}, f_{i+2}, \dots)$, then \[ {\emph{mdim}_M}(X,\textit{\textbf{f}}\,)={\emph{mdim}_M}(Y,\textit{\textbf{g}}). \]
\end{corollary}
\begin{proof} Note that the following diagram is commutative
\[
 \begin{CD}
    X @>f_1>> X@>f_i>>\dots X @>f_{i}>> X @>f_{i+1}>>X @>f_{i+2}>>X \\
    @V{h_{1}}VV      @VV{h_{2}}V     @VV{h_{i}}V  @VV{Id_{X}}V    @VV{Id_{X}}V     @VV {Id_{X}}V\\
   X @>g_1>> X@>g_i>>\dots  X @>g_{i}>>X @>f_{i+1}>>X @>f_{i+2}>>X
  \end{CD}
\]
where $I_{X}$ is the identity of $X$ and $h_{i}=g_{i}^{-1}\circ f_{i}$, $h_{i-1}=g_{i-1}^{-1}\circ h_{i}\circ f_{i-1}$, \dots, $h_{1}=g_{1}^{-1}h_{2}f_{1}$.  Furthermore, $(h_{1}, h_{2}, \dots, h_{i}, I_{X}, I_{X} ,\dots )$ is an equicontinuous sequence of homeomorphisms. Therefore, $\textit{\textbf{f}}$ and $\textit{\textbf{g}}$ are uniformly equiconjugate. The corollary follows from Theorem \ref{edee344}, since the infimum $\inf_n \{d(h^{-1}_n(y_1),h_n^{-1}(y_2)),d(h_n(x_1),h_n(x_2))\}>0$ is taken over a finite set.
\end{proof}

Next corollary    means that if $\textit{\textbf{f}}$ is a sequence of homeomorphisms then  the metric mean dimension  is independent on the firsts elements in the sequence $\textit{\textbf{f}}.$

\begin{corollary}\label{corolarioigualdad}  Let $\textit{\textbf{f}}=(f_{n})_{n=1}^{\infty}$ be a non-autonomous dynamical system consisting of homeomorphisms. For any $i,  j\in \mathbb{N}$ we have  $$ {\emph{mdim}_M}(X,\sigma^i(\textit{\textbf{f}}\,))=   {\emph{mdim}_M}(X,\sigma^j(\textit{\textbf{f}}\,)).$$
\end{corollary}
\begin{proof} It is sufficient to prove that $ {\text{mdim}_M}(X,\sigma^i(\textit{\textbf{f}}\,))=   {\text{mdim}_M}(X,\textit{\textbf{f}}\,)$ for all $i\in\mathbb{N}$. Fix $i\in\mathbb{N}$.  Take     $\textit{\textbf{g}}=(g_{n})_{n\in \mathbb{N}}$, where, for each $n\leq i$,   $g_{n}=I$ is the identity on $X$  and $g_{n}=f_{n}$ for $n>i$.   It follows from Corollary \ref{newe} that $$ {\text{mdim}_M}(X, \textit{\textbf{f}}\,)=   {\text{mdim}_M}(X, \textit{\textbf{g}}).$$ For each $x,y\in X$ and $n> i$  we have \begin{align*}\max \{d(x,y),\dots , d(g_{1}^{(i-1)}(x), &g_{1}^{(i-1)}(y)), \dots, d(g_{1}^{(n-1)}(x), g_{1}^{(n-1)}(y)) \}\\
&= \max \{d(x,y), d(g_{i}(x), g_{i}(y)), \dots, d(g_{i}^{(n-i)}(x), g_{i}^{(n-i)}(y)) \}
\\
&= \max \{d(x,y), d(f_{i}(x), f_{i}(y)), \dots, d(f_{i}^{(n-i)}(x), f_{i}^{(n-i)}(y)) \}.
\end{align*}
 Hence $$ {\text{mdim}_M}(X,\textit{\textbf{f}}\,)= {\text{mdim}_M}(X,\textit{\textbf{g}})=   {\text{mdim}_M}(X,\sigma^i(\textit{\textbf{f}}\,)),$$ which proves the corollary.
\end{proof}


Next corollary follows from  Corollary \ref{corolarioigualdad} and Proposition \ref{propo211} (see the proof of Corollary \ref{desdede}). 
\begin{corollary}  For any homeomorphisms  $f$ and $g$ defined on $X$, we have  $$ {\emph{mdim}_M}(X,f\circ g)=   {\emph{mdim}_M}(X,g\circ f).$$ 
\end{corollary}


\section{On the  continuity of the metric mean dimension}\label{section7}

In this section we will show some   results related to the continuity of the metric mean dimension of sequences of diffeomorphisms defined on a manifold.  For any $r\geq 0,$ set  
\[ \mathcal{C}^{r}(X)=\{ (f_{n})_{n=1}^{\infty}: f_{n}:X\rightarrow X \text{ is a }C^{r}\text{-map}\}=\prod_{i=1} ^{+\infty}  \text{C}^{r}(X) ,\]
where $\text{C}^{r}(X)=\{\phi:X\rightarrow X: \phi\text{ is a }C^{r}\text{-map}\}\footnote{If $r\geq 1$ we assume that $X$ is a Riemannian manifold}.$ 
Hence $ \mathcal{C}^{r}(X)$ can be endowed with the \textit{product topology}, which   is generated by the sets
\[ \mathcal{U}=\prod_{i=1} ^{j}  \text{C}^{r}(X)\times \prod_{i=j+1}^{j+m}U_{i} \times \prod_{i>j+m} ^{+\infty}  \text{C}^{r}(X), \]
where $U_{i}$ is an open subset of $\text{C}^{r}(X)$, for $j+1\leq i\leq j+m,$ for some  $j,m\in \mathbb{N}$.
The space $\mathcal{C}^{r}(X)$ with the product topology will be denoted by $(\mathcal{C}^{r}(X),\tau_{prod}).$ We can consider the map 
\begin{align*}
    \underline{\text{mdim}_{M}}:(\mathcal{C}^{r}(X),\tau_{prod})&\rightarrow \mathbb{R}\cup \{+\infty\}\\
    \textit{\textbf{f}} &\to  \underline{\text{mdim}_{M}}(\textit{\textbf{f}},X). 
\end{align*}

Clearly, if $ \underline{\text{mdim}_{M}}$ is a constant map, then is continuous. 

\begin{proposition}  If $\underline{\emph{mdim}_{M}}:(\mathcal{C}^{r}(X),\tau_{prod})\rightarrow \mathbb{R}\cup \{+\infty\}$ is not constant then is discontinuous at any $\textit{\textbf{f}}\in \mathcal{C}^{r}(X)$.
\end{proposition}
\begin{proof} Fix $ \textit{\textbf{f}}=(f_{n})_{n=1}^{\infty}\in \mathcal{C}^{r}(X)$.   Since $\underline{\text{mdim}_{M}}$  is not constant, there exists  $ \textit{\textbf{g}}=(g_{n})_{n=1}^{\infty}\in   \mathcal{C}^{r}(X) $ such that
$\underline{\text{mdim}_{M}}(X,\textit{\textbf{g}})\neq \underline{\text{mdim}_{M}}(X,\textit{\textbf{f}}\,).$ Let $\mathcal{V}\in \tau_{prod}$ be any open neighborhood of $\textit{\textbf{f}}$. For some $k\in \mathbb{N}$, the sequence  $\textit{\textbf{j}}=(j_{n})_{n=1}^{\infty}$, defined by
\begin{equation*}j_{n}=
\begin{cases}
  f_{n} & \mbox{if }   n=1,\dots, k \\
  g_{n} & \mbox{if } n>k , \\
        \end{cases}
\end{equation*}
belongs to $\mathcal{V}$, by definition of $\tau_{prod}$. It is follow from Corollary \ref{newe} that $ \underline{\text{mdim}_{M}}(X,\textit{\textbf{j}})= \underline{\text{mdim}_{M}}(X,\textit{\textbf{g}}).$
which proves the proposition.
\end{proof}

Let $d^{1}(\cdot,\cdot)$ be a $C^{1}$-metric on $\text{C}^{1}(X)$. Suppose that $\sup_{n\in\mathbb{N}}\Vert D f_{n}\Vert <\infty$. For any $K>0$, if $d^{1}(g_{n},f_{n})<K$, then $\sup_{n\in\mathbb{N}}\Vert D g_{n}\Vert <\infty$ and therefore $  \underline{\text{mdim}_M}(X,\textit{\textbf{g}},d)=0.$ On the other hand, if $\sup_{n\in\mathbb{N}}\Vert D f_{n}\Vert =\infty$, then $  \underline{\text{mdim}_M}(X,\textit{\textbf{f}},d)$ is not necessarily zero.

\medskip

In \cite{JeoCTE}, Section 6,  is proved that: 
 \begin{proposition}\label{escolhari} If  $\textit{\textbf{f}}=(f_{n})_{n=1}^{\infty} $ is a sequence of $C^{1}$-diffeomorphisms, there exists a sequence of positive numbers $(\delta_{n})_{n=1}^{\infty}$   such that every sequence of diffeomorphisms  $\textit{\textbf{g}}=(g_{n})_{n=1}^{\infty}$ with  \(d^{1}(f_{n},g_{n})<\delta_{n}\) for each $n\geq 1,$
 is uniformly equiconjugate to $\textit{\textbf{f}}$ by a sequence $(h_{n})_{n=1}^{\infty}$ such that $h_{n}\rightarrow I_{X}$ as $n\rightarrow \infty$.
 \end{proposition}
  
 Note that, if $h_{n}\rightarrow I_{X}$ as $n\rightarrow \infty$, then for any $x_{1}\neq x_{2}\in X$ and $y_{1}\neq y_{2}\in Y$ we have $\inf_n \{d(h^{-1}_{n}(y_1),h_{n}^{-1}(y_2)),d(h_{n}(x_1),h_{n}(x_2))\}>0$. Hence, it follows from Theorems \ref{edee344} and Proposition  \ref{escolhari} that
\begin{corollary}\label{esded} Given  a sequence of diffeomorphisms  $\textit{\textbf{f}}=(f_{n})_{n=1}^{\infty}$, there exists a sequence of positive numbers $(\delta_{n})_{n=1}^{\infty}$   such that if $\textit{\textbf{g}}=(g_{n})_{n=1}^{\infty} $ is a sequence of diffeomorphisms such that   \(d^{1} (f_{n},g_{n})<\delta_{n}\) for each $n\geq 1,$  then \begin{equation*}
\underline{\emph{mdim}_M}(X,\textit{\textbf{g}})=\underline{\emph{mdim}_M}(X,\textit{\textbf{f}}\,).
\end{equation*}
\end{corollary}

Roughly, Corollary \ref{esded} means that if $d ^{1}(f_{n},g_{n}) $ converges very quickly to zero as $n\rightarrow \infty$, then  \begin{equation*}
\underline{\text{mdim}_M}(X,\textit{\textbf{f}}\, )=\underline{\text{mdim}_M}(X,\textit{\textbf{g}}).
\end{equation*}

  For each sequence of diffeomorphisms  $\textit{\textbf{f}}=(f_{n})_{n=1}^{\infty}$ and a sequence of positive numbers   $ \varepsilon=(\varepsilon_{n})_{n=1}^{\infty}$, a  \textit{strong basic neighborhood} of $\textit{\textbf{f}}$ is the set \[ B^{r}(\textit{\textbf{f}},\varepsilon)=  \left\{\textit{\textbf{g}}=(g_{n})_{n=1}^{\infty}:  g_{n} \text{ is a }C^{r}\text{-diffeomorphism and }  d  (f_{n},g_{n}) <\varepsilon_{n},\text{ for all  }  n\in\mathbb{N}\right\}. \]
 The  \textit{strong topology} (or \textit{Whitney topology}) on $\mathcal{C}^{r}(X)$ is generated by the strong basic neighborhoods  of each $\textit{\textbf{f}}\in \mathcal{C}^{r}(X)$. The space $\mathcal{C}^{r}(X)$ with the strong topology will be denoted by $(\mathcal{C}^{r}(X),\tau_{str}).$

\begin{corollary}\label{teoprinc} For    $r\geq 1$, let  $\mathcal{D}^{r}(X)\subseteq \mathcal{C}^{r}(X)$ be the set consisting of diffeomorphisms.   Then 
\[\underline{\emph{mdim}_M}:(\mathcal{D}^{r}(X),\tau_{str})\rightarrow \mathbb{R}\cup \{+\infty\}   \]
is a continuous map. 
\end{corollary}
\begin{proof} Let $\textit{\textbf{f}}\in \mathcal{D}^{r}(X)$. If follows from  Theorem \ref{escolhari}  that there exists a strong basic neighborhood $B ^{r}(\textit{\textbf{f}}, (\delta_{n})_{n=1}^{\infty})$ such that  every  $\textit{\textbf{g}}\in B ^{r}(\textit{\textbf{f}}, (\delta_{n})_{n=1}^{\infty})$ is   uniformly equiconjugate to $\textit{\textbf{f}}$. Thus, from  Proposition \ref{edee344} we have $\underline{\text{mdim}_M}(X,\textit{\textbf{g}})=\underline{\text{mdim}_M}(X,\textit{\textbf{f}}\,)$   for all $\textit{\textbf{g}}\in B^{r}(\textit{\textbf{f}}, (\delta_{n})_{n=1}^{\infty})$, which proves the corollary.
\end{proof}

 A real valued function $\varphi : X \rightarrow \mathbb{R}\cup \{\infty\}$ is called \textit{lower} (respectively \textit{upper})  \textit{semi-continuous on a point} $x\in X$ if $$\liminf_{y\rightarrow x}\varphi (y)\geq \varphi (x)\quad  (\text{repectively } \limsup_{y\rightarrow x}\varphi (y)\leq \varphi (x) ).    $$    $\varphi  $ is called \textit{lower} (respectively \textit{upper})  \textit{semi-continuous} if is  lower (respectively  {upper})  {semi-continuous on any point} of $ X$.

\begin{remark} From now on, we will consider   $\tilde{X}=[0,1]$ or $\mathbb{S}^1$. \end{remark}

Misiurewicz in \cite{Misiurewicz}, Corollary 1, proved that $h_{top}: C^{0}([0,1])\rightarrow \mathbb{R}\cup \{\infty\}$ is lower semi-continuous. For the case of the metric mean dimension we have:  
 
 \begin{proposition}\label{hfjdjehr}  $\emph{mdim}_{M}:C^{0}(\tilde{X})\rightarrow \mathbb{R}$ is nor lower neither upper  semi-continuous on maps with metric mean dimension in $(0,1)$. Furthermore,   
 $\emph{mdim}_{M}:C^{0}(\tilde{X})\rightarrow \mathbb{R}$ is not  lower   semi-continuous on maps with metric mean dimension in $(0,1]$ and is not  upper   semi-continuous on maps with metric mean dimension in $[0,1)$. 
 \end{proposition}
 \begin{proof}
 Let $\varphi$ be a continuous map on $\tilde{X}$.  If $\text{mdim}_{M}(\varphi)=1$, we can    approximate $\varphi$ by a continuous map with zero metric mean dimension (take a sequence of $C^{1}$-maps converging to $\varphi$). Next, suppose that $\text{mdim}_{M}(\varphi)=0$. Firstly, take $\tilde{X}=[0,1]$. Fix  $\varepsilon >0$.  
  Let $p^{\ast}$ be a fixed point of $\varphi$.   Choose  $\delta>0$ such that $d(\varphi(x),\varphi (p^{\ast}))<\varepsilon/2$ for any $x$ with $d(x,p^{\ast})<\delta$.  Let $\phi $ and $T_{2} $ be as in   Example \ref{exfagner}, with $J_{1}=[0,p^{\ast}]$, $J_{2}=[p^{\ast},p^{\ast}+\delta/2]$, $J_{3}=[p^{\ast}+\delta/2,p^{\ast}+\delta] $ and $J_{4}=[p^{\ast}+\delta,1]$. Take the continuous map $\psi$ on $[0,1]$ defined as $$
\psi(x)= \begin{cases}
    \varphi(x), &  \text{ if }x\in J_{1}\cup J_{4}, \\
   T_{2}^{-1}\phi T_{2}(x), &  \text{ if }x\in J_{2}, \\
    \psi_{1}(x), &  \text{ if }x\in J_{3}, 
      \end{cases}
$$ where $ \psi_{1}$ is the  affine map on $J_{3}$ such that $ \psi_{1}(p^{\ast}+\delta/2)=(p^{\ast}+\delta/2)$  and $ \psi_{1}(p^{\ast}+\delta)=\varphi(p^{\ast}+\delta).$ Note that $d(\psi,\varphi)<\varepsilon.$ It follows from Proposition \ref{invariant} that  $$\text{mdim}_{M}(\tilde{X},\psi,|\cdot |)=\max \{\text{mdim}_{M}(\tilde{X},\psi|_{J_{1}\cup J_{3}\cup J_{4}},|\cdot |)  ,\text{mdim}_{M}(\tilde{X},\psi|_{J_{2}},|\cdot |)\}= 1,$$ 
since $\text{mdim}_{M}(\tilde{X},\psi|_{J_{1}\cup J_{3}\cup J_{4}},|\cdot |)\leq \text{mdim}_{M}(\tilde{X},\varphi,|\cdot |) =0$.
Analogously we can prove that any $\varphi\in C^{0}([0,1])$ with metric mean dimension  in $(0,1)$ can be approximated by both a continuous map with metric mean dimension equal to 1 and a continuous map with metric mean dimension equal to 0. These facts prove the proposition for $\tilde{X}=[0,1]$.  For $\tilde{X}=\mathbb{S}^{1}$, we can approximate any $\varphi\in C^{0}(\mathbb{S}^{1})$ by a map $\varphi^{\ast}$ with periodic points. We can prove analogously  that $\varphi^{\ast}$ can be approximate by a continuous map on $\mathbb{S}^{1}$ with metric mean dimension equal to 0 or equal to 1, which proves the proposition for $\tilde{X}=\mathbb{S}^{1}$. 
 \end{proof}

Next, Kolyada and Snoha in \cite{K-S}, Theorem F, showed that  $h_{top}:\mathcal{C}([0,1])\rightarrow \mathbb{R}\cup\{\infty\}$ is not lower semi-continuous, endowing  $\mathcal{C}([0,1])$   with the metric $$D((f_{n})_{n=1}^{\infty},(g_{n})_{n=1}^{\infty})=\sup_{n\in\mathbb{N}}\max_{x\in [0,1]}|f_{n}(x)-g_{n}(x)|.$$ Furthermore, they proved in Theorem G that $h_{top}:\mathcal{C}([0,1])\rightarrow \mathbb{R}\cup\{\infty\}$ is  lower semi-continuous on any constant sequence $(\phi, \phi,\dots)\in \mathcal{C}(\tilde{X})$.
However, It follows from Proposition \ref{hfjdjehr} that:
 
 \begin{corollary}   $\emph{mdim}_{M}:\mathcal{C}(\tilde{X})\rightarrow \mathbb{R}$ is nor lower neither upper  semi-continuous on any constant sequence $(\phi, \phi,\dots)\in \mathcal{C}(\tilde{X})$. Consequently, $\emph{mdim}_{M}:\mathcal{C}(\tilde{X})\rightarrow \mathbb{R}\cup\{\infty\}$ is nor lower neither upper  semi-continuous.  
 \end{corollary} 
 
  Take $\textbf{\textit{f}}=(f_{n})_{n=1}^{\infty}$ on $\tilde{X}$ defined by $f_{n} =\psi^{2^{n}}$ for each $n\in\mathbb{N}$, where $\psi$ is the map from Example \ref{exfagner}. We have  $\text{mdim}_{M}(\tilde{X},\textbf{\textit{f}},|\cdot|)=\infty$ (see Example \ref{hfkenrkflr}). Thus  there exist non-autonomous dynamical systems on $\tilde{X}$ with infinite metric mean dimension. Consequently $\text{mdim}_{M}:\mathcal{C}(\tilde{X})\rightarrow \mathbb{R}\cup \{\infty\}$ is unbounded. 

\medskip

 We finish this work with the next result:
 \begin{theorem}\label{bnfuefnf}
$\emph{mdim}_{M}:\mathcal{C}(\tilde{X})\rightarrow \mathbb{R}\cup\{\infty\}$ is not lower semi-continuous on any non-autonomous dynamical system  with non-zero metric mean dimension.  
\end{theorem}  
\begin{proof}
Let $\textit{\textbf{f}}=(f_{n})_{n=1}^{\infty}$ be a non-autonomous dynamical system with positive metric mean dimension. Let $\lambda_{m}$ be a sequence in $[0,1]$ such that $\lambda_{m}\rightarrow 1$ and $\lambda_{m}\cdots \lambda_{1}\rightarrow 0$ as $m\rightarrow \infty$. Take $\textbf{\textit{g}}_{m}=(\lambda_{m+n}f_{n})_{n=1}^{\infty}$. Thus $\textbf{\textit{g}}_{m}\rightarrow \textbf{\textit{f}}$ as $m\rightarrow \infty$. However, for any $x\in \tilde{X}$, $({g}_{m})^{(k)}(x)\rightarrow 0$ as $k\rightarrow \infty$. Consequently, the metric mean dimension of $\textbf{\textit{g}}_{m}$ is zero for each $m\in \mathbb{N}$. 
\end{proof}

\end{document}